\documentclass[11pt,a4paper,reqno]{amsart}

\usepackage{a4wide}

\usepackage{amsmath,mathrsfs,dsfont,amsfonts,amssymb,amsthm,stmaryrd,mathtools,relsize,scalerel,graphicx,fancyhdr}

\allowdisplaybreaks  

\usepackage{hyperref}
\usepackage[T2A,T1]{fontenc}
\DeclareSymbolFont{cyrillic}{T2A}{cmr}{m}{n}
\DeclareMathSymbol{\DD}{\mathalpha}{cyrillic}{196}

\newcommand{\Pre}{\operatorname{Pre}}
\newcommand{\Mis}{\mathcal{MIS}}
\newcommand{\sgn}{\operatorname{sgn}}
\newcommand{\leb}{\operatorname{Leb}}
\newcommand{\dd}{\mathrm{d}}
\newcommand{\supp}{\operatorname{supp}}
\newcommand{\isEquivTo}[1]{\underset{#1}{\sim}}

\def\R{\ensuremath{\mathds R}}
\def\N{\ensuremath{\mathds N}}

\def\RT{r} 
\def\ie{{\em i.e.}, }

\theoremstyle{plain}
\newtheorem{thm}{Theorem}[section]
\newtheorem{lem}[thm]{Lemma}
\newtheorem{cor}[thm]{Corollary}
\newtheorem{prop}[thm]{Proposition}
\theoremstyle{definition}
\newtheorem{defn}[thm]{Definition}

\newtheorem{rem}[thm]{Remark}
\newtheorem*{theorem-non}{Notation}

\numberwithin{equation}{section}

\usepackage{enumitem}
\makeatletter
\def\namedlabel#1#2{\begingroup
	#2%
	\def\@currentlabel{#2}%
	\phantomsection\label{#1}\endgroup
}

\usepackage{xcolor}

\definecolor{unbleu}{rgb}{0.03, 0.15, 0.4}

\hypersetup{
pdfborder = {0 0 0},
colorlinks,
linkcolor=unbleu,
citecolor=unbleu,
urlcolor=unbleu
}

\begin{document}

\title[Inducing techniques for quantitative recurrence and applications]{Inducing techniques for quantitative recurrence and applications to Misiurewicz maps and doubly intermittent maps}	
	
\author[D. Bansard-Tresse]{Dylan Bansard-Tresse}
\address{Dylan Bansard-Tresse\\ CPHT, \'Ecole Polytechnique, Institut Polytechnique de Paris\\ 91128 Palaiseau cedex \\ France} 
\email{\href{mailto:dylan.bansard-tresse@polytechnique.edu }{dylan.bansard-tresse@polytechnique.edu}}

\author[J. M. Freitas]{Jorge Milhazes Freitas}
\address{Jorge Milhazes Freitas\\ Centro de Matem\'{a}tica \& Faculdade de Ci\^encias da Universidade do Porto\\ Rua do Campo Alegre 687\\ 4169-007 Porto\\ Portugal}
\email{\href{mailto:jmfreita@fc.up.pt}{jmfreita@fc.up.pt}}
\urladdr{\url{http://www.fc.up.pt/pessoas/jmfreita/}}

\thanks{DBT and JMF were partially supported by FCT projects PTDC/MAT-PUR/28177/2017, PTDC/MAT-PUR/4048/2021 and 2022.07167.PTDC, with national funds, and by CMUP, which is financed by national funds 
through FCT -- Funda\c{c}\~ao para a Ci\^encia e a Tecnologia, I.P., under the project with reference UIDB/00144/2020. The authors would like to thank Romain Aimino, Th\'eophile Caby, Jean-Ren\'e Chazottes, Mike Todd and Roland Zweim\"uller for fruitful and stimulating conversations about the results in the paper.}
	
\date{\today}
\keywords{Inducing techniques, point processes, periodic points, clustering, return times, hitting times, compound Poisson process} 
\subjclass[2020]{37A50, 37B20, 60G70, 60G55,  37A25}


\begin{abstract}
We prove an abstract result establishing that one can obtain the convergence of Rare Events Point Processes counting the number of orbital visits to a sequence of shrinking target sets from the convergence of corresponding point processes for some induced system and matching shadowing shrinking sets inside the base of the inducing scheme. We apply this result to prove a dichotomy for two classes of non-uniformly hyperbolic interval maps: Misiurewicz quadratic maps and doubly intermittent maps. The dichotomy holds in the sense that the shrinking target sets may accumulate in any individual point $\zeta$ chosen in the phase space and then one either obtains a limiting homogeneous Poisson process at every non-periodic point $\zeta$ or a limiting compound Poisson process with geometric multiplicity distribution at every periodic point. We also highlight the reconstruction performed in order to recover the multiplicity distribution for a periodic orbit sitting outside the base of the induced map.   
\end{abstract}
		
\maketitle

	
\tableofcontents

	\section{Introduction}
	
Inducing techniques are very powerful tools to study the statistical properties of dynamical systems. Among these properties we are particularly interested in the extremal behaviour which is tied to  the study of quantitative recurrence to shrinking target sets in the phase space. 
The idea of using induced systems to study hitting and return times statistics appeared first in the insightful paper \cite{BSTV03}, where it was shown that, for sequences of nested balls shrinking to \emph{a.e.} point in the base of the induced system, the existence of a limiting law for the normalised return times (or hitting times) to these balls for the induced dynamics implied that the same limiting law applied for the original dynamics. This allowed them to derive hitting and return times statistics for 	non-uniformly hyperbolic interval maps such as maps with critical points or neutral fixed points, which admit induced systems with good hyperbolic properties, for which an exponential limiting law was easy to derive. We note that the existence of a limiting law for the normalised return time is equivalent to the existence of a limiting law for the normalised hitting time (in which case the orbits may not necessarily start in the target sets) and the two limits are related by an integral equation which has the standard exponential distribution as a fixed point (see \cite{HLV05}). We also observe that the existence of limiting law for the hitting time to a nested sequence of balls shrinking to a certain point $\zeta$ is equivalent to the existence of a distributional limit for the partial maxima of a stochastic process for which the observation of exceedances of high levels corresponds to the entrance in small balls around $\zeta$ (see \cite{FFT10}).

The fact that the induced and the original system shared the same limiting laws for the hitting/return times was generalised in \cite{HWZ14} so that any point could be taken as the intersection of the nested sequence of balls (instead of only typical points). This was further generalised in \cite{FFTV16}, where again the connection between the induced dynamics and the original one was established regarding the convergence of Rare Events Point Processes (REPP), which keep information not only of the first hitting/return time but also of all succeeding hits/returns. In simple terms, these point processes count the number of visits to the chosen sequence of shrinking (hence rarer) target sets.

Earlier results established that for well behaved systems, for \emph{a.e.} point $\zeta$ chosen in the phase space and for a nested sequence of balls (or cylinders) shrinking to $\zeta$, we had exponential hitting and return times statistics (meaning that the limiting law for both the normalised hitting and return time is the standard exponential distribution). On the other hand, for special points like when $\zeta$ is a periodic point, the limiting law for the hitting time is exponential with parameter $0<\theta<1$, while the law for the return time is a mixture of an exponential distribution with the same parameter and a discrete component placing a mass point at $0$ with weight $1-\theta$. This was deeply studied in \cite{FFT12}, where the periodicity of $\zeta$ was associated with the occurrence 
of clustering of rare events, so that a visit to a vicinity of $\zeta$ would usually mean the appearance of a cluster of succeeding visits, which was responsible for the mass point at $0$ observed in the return times statistics. The parameter $\theta$ measured the intensity of clustering and, following the classical Extreme Value Theory, was called Extremal Index. Moreover, in \cite{FFT12}, it was actually proved that for a uniformly hyperbolic system a dichotomy held: either we had this mixture for the return times statistics at every periodic point or we would have a standard exponential limiting law at every non-periodic point (which means $\theta=1$), with no exceptions. This dichotomy was conjectured to held in much more generality and, later, it was established for some uniformly expanding interval maps with a finite number of branches in \cite{FP12}, for maps for which there was a spectral gap for the respective transfer operators in \cite{K12} or maps with a strong form of decay of correlations in \cite{AFV15}, which included Rychlik maps (\cite{R83}), with possibly countably many branches. Nonetheless, this meant that the dichotomy held essentially for nicely expanding systems. 

In \cite{FFTV16}, using the inducing technique, the authors managed to prove the dichotomy for non-uniformly expanding systems with a neutral fixed point. It was established in terms of the convergence of REPP, whose limits were a standard homogeneous Poisson process, for all non-periodic points, and a compound Poisson process, for periodic points. The compound Poisson process could be described has having two components, one was the time positions of the clusters of rare observations, scattered in the time line according to a homogeneous Poisson process of intensity $\theta$, and the other was the Geometric multiplicity distribution of parameter $\theta$ describing the cluster sizes.  
	
One of the key aspects of the argument used in \cite{FFTV16} in order to obtain the full dichotomy, was the fact that the special structure of the Liverani-Saussol-Vaienti (LSV) maps \cite{LSV99} allowed the authors to choose different bases for the inducing scheme so that every point (except for the neutral fixed point $0$, which was analysed separately) could be covered by one of these bases. Recall that all the results mentioned above regarding the connection between the limiting laws for the induced and the original dynamics assumed always that the point $\zeta$ which is the accumulation point of the targets sets must be in the base of the induced map. In \cite{Zwe18},  Zweim\"uller managed to remove this obstruction and proved that the induced map shared the same hitting times statistics, for some shadow shrinking sets inside the base, with that of the original map, for which one considered a sequence of original shrinking target sets outside the base. This is an abstract quite general result which holds as long as the time needed to get from the shadow sets inside the base to the original target sets is negligible when compared with the expected time to return to the latter. The dichotomy for LSV maps can then be proved using only the usual induced map with base $[1/2,1]$ and even the case of the neutral fixed point can be covered by analysing its preimage  inside the base ($1/2$), as was done in \cite{Zwe18}. We also mention the paper \cite{DT21} where the authors use inducing techniques to study systems with holes outside the base. 
	
One of the main goals of this paper is to generalise Zweim\"uller's abstract result for hitting times statistics to the convergence of REPP, \ie we to establish that one can obtain the convergence of REPP for shrinking target sets outside the base from the convergence of corresponding point processes for the induced dynamics and shadowing sets inside the base. We remark that this is not straightforward because one must guarantee that the induced system does not miss clusters, which means that not only must one go fast enough from the shadow sets to the original targets as one should return to the base of the inducing scheme before returning to the target sets since, otherwise, the induced map is missing part of the action. This is carried in Section~\ref{sec:abtract-result}.	

We then apply these abstract results to prove a dichotomy regarding the convergence of REPP for two classes of non-uniformly hyperbolic interval maps. In Section~\ref{sec:Misiurewicz}, we consider Misiurewicz-Thurston quadratic maps, for which the critical point is pre-periodic and have the nice property of admitting a Rychlik induced system. 
In Section~\ref{sec:doubly-intermittent}, we prove the dichotomy for the class of doubly intermittent maps introduced very recently in \cite{CLM22}, which also admit nice induced systems.	

We observe that the induced map may not detect the clustering visits of the orbits to the target sets outside the base because these may occur before the orbits return to the base. Therefore, one must reconstruct the original point process counting all the visits to the target sets from the point process of visits to the shadowing sets inside the base. This reconstruction must be carried using the local behaviour of the dynamics in the limit of the shrinking target sets. We perform this reconstruction for target sets shrinking to a periodic point whose orbit never enters the base of the inducing scheme. This was done in Section~\ref{reconstruction_cluster} and we believe that it has an independent interest on its own.

		\section[Generalised inducing technique for the convergence of REPP]{Generalised inducing technique for the convergence of Rare Events Point Processes}
		\label{sec:abtract-result}

	Let $(\mathcal{X}, \mathscr{B},\mu,T)$ be an ergodic dynamical system where $\mathcal{X}$ is a metric space, $\mathscr{B}$ is the Borel $\sigma$-algebra on it and $T$ is a map preserving the probability measure $\mu$. 
	Let $A \in \mathscr{B}$ with $\mu(A) > 0$.  Given $x \in \mathcal{X}$, the first hitting time to $A$ as
	\[
	\RT_A(x) = \inf \big\{k\geq 1 \colon  T^kx \in A\big\}.
	\]
	When $x\in A$ we say that $r_A$ is the first return time to $A$. For all $i> 1$,  we define the $i$-th hitting/return time to $A$ inductively by 
	\[
	\RT_A^{(i)}=\RT_A\Big(T^{\RT_A^{(i-1)}(x)}(x)\Big)
	\]
	and by convention $\RT_A^{(1)}=\RT_A$.
	If $\RT_A^{(i)} = +\infty$ for some $i$, then we set $\RT_A^{(j)} = +\infty$ for $j \geq i$. The induced map $T_A\colon A\to A$ is defined by
	\[
	T_A(x) = T^{\RT_A(x)}(x).
	\] 
	This map is well defined $\mu$-almost everywhere by Poincar\'e's recurrence theorem. Then, the induced dynamical system $\big(A, T_A, \mu_A\big)$ is also an ergodic dynamical 
	system, where $\mu_A(B)=\mu(A\cap B)/\mu(A)$ for $B\in \mathscr{B}$, and $ \mathscr{B}_A =\{B \cap A : B\in \mathscr{B}\}$.  In this case, for $B \subset A$ and $x\in A$, we define the induced hitting/return times by
	\[
	 \RT^A_B(x) = \inf \big\{k\geq 1 \colon T_A^kx \in B\big\}.
	\]
	We then define the successive hitting/return times $\RT^{A,(i)}_B$ for $T_A$ in the same way as for $T$.
	
\begin{defn}	
	For $B\in \mathscr{B}$, we define the stationary process of successive hitting/return times: 
	$$\Phi_B = \big(\RT_B^{(1)}, \RT_B^{(2)}, \dots\big),$$ where the $\RT_B^{(k)}$'s are random variables with the 
	same law which is induced by $\mu$ ({\em i.e.}, 
	$\mathds{P}\big(\RT_B^{(1)}=\cdot\big)=\mathds{P}\big(\RT_B^{(k)}=\cdot\big)=\mu\big(\big\{x\in\mathcal{X}: \RT_B^{(1)}(x)=\cdot\big\}\big)$ for all $k\geq 2$). Given $B 
	\subset A$, we can define 
	the induced process on $A$ by 
	$$\Phi^A_B = \big(\RT^{A,(1)}_B, \RT^{A,(2)}_B, \dots\big),$$ which is stationary with respect to $\mu_A$. 
\end{defn}	
	\begin{defn}
	We define the Rare Event Point Processes (REPP), on $\R_0^+$, which count the number of orbital visits to the set $B\in\mathcal B$, in a normalised time frame, for the original and the induced dynamics in the following way:
	\begin{equation*}
			\mathcal{N}_B(x) = \sum_{i\geq 0} \delta_{i\cdot\mu(B)} \cdot \mathbf{1}_{B}(T^ix),\;\text{for $x\in\mathcal X$},\; \mathcal{N}^A_B(x) = \sum_{i\geq 0} \delta_{i\cdot\mu(B)} \cdot \mathbf{1}_{B}(T_A^ix), \; \text{for $x\in A$},
		\end{equation*}
		where $\delta_{z}$ denotes the Dirac measure charging the mass point $z\in\R_0^+$ and, in the second case, we have $B\subset A$.
\end{defn}
\begin{rem}
Observe that the components of the process $\mu(B)\cdot\Phi_B$ (respectively $\mu_A(B)\cdot\Phi_B^A$) correspond to the interarrival times of the projection of  the point process $\mathcal{N}_B$ (respectively $\mathcal{N}^A_B$) to the space of continuous time c\`adl\`ag stochastic processes, or, in other words, the sequences
$$\left(\mu(B)\sum_{i=1}^j \RT_B^{(i)}\right)_{j\in\N} \quad\text{and}\quad \left(\mu(B)\sum_{i=1}^j\RT_B^{A,(i)}\right)_{j\in\N}$$ correspond to the sequence of mass points charged by the point process $\mathcal{N}_B$ and its induced version $\mathcal{N}^A_B$, respectively. For this reason, we will refer to the processes $\mu(B)\cdot\Phi_B$ and $\mu_A(B)\cdot\Phi_B^A$ as the normalised interarrival times process and induced interarrival times process, while $\Phi_B$ and $\Phi_B^A$ will be referred to as the unnormalised interarrival times and induced interarrival times processes.  
\end{rem}

	We will study the convergence of REPP when the measure of the target sets $B$ shrinks to $0$, which motivates the following definition.
	
	\begin{defn}[Asymptotically rare events]
		We say that a sequence of measurable sets $(E_n)_{n\in \mathbb{N}}\subset \mathscr{B}$ is asymptotically rare if $\mu(E_n) \to 0$ when $n\to +\infty$.
	\end{defn}	
	
For a sequence $(E_n)_{n\in \mathbb{N}}$ of asymptotically rare events, we can define the sequence of REPP $(N_n)_{n\in\N}$ and when $E_n\subset A$, for all $n\in\N$, the sequence of induced REPP $(N^A_n)_{n\in\N}$ by:
$$	
N_n(x):=\mathcal N_{E_n}(x),\;\text{for $x\in\mathcal X$},\quad N^A_n(x):=\mathcal N^A_{E_n}(x),\;\text{for $x\in\mathcal A$}.
$$	
\begin{rem}
As observed in \cite[Remark~3.5]{Z22}, using the continuous mapping theorem, one can show that the weak convergence of the normalised interarrival times process $\mu(B_n)\cdot\Phi_{B_n}$ (respectively $\mu_A(B_n)\cdot\Phi^A_{B_n}$) implies the weak convergence of the point process $N_n$ (respectively $N^A_n$) on the space of Radon point measures equipped with the vague topology (see \cite[Chapter~3]{R08}). 
\end{rem}

One of the key ideas to recover the information regarding visits to sets outside the base of the induced map is to consider their respective shadows in the base. Hence, we introduce the following notion. 

	\begin{defn}[Shadow set]
		Let $A \in \mathscr{B}$ with $\mu(A) > 0$. For every $E \in \mathscr{B}$, its shadow set $E'_A$ in $A$  is \\
		\begin{equation}
			\label{eq:B-prime-definition}
			E'_A = \bigcup_{k\geq 0} A \cap \{\RT_A > k \} \cap T^{-k}(E).
		\end{equation}
	\end{defn}
	In the sequel, $A$ will be a set on which we will induce. To alleviate notation, we shall simply write $E'$ instead of $E'_A$ when it is clear from the context that we are inducing on $A$.

		\begin{theorem-non}
		We write $\xRightarrow[]{\;\mu\;\,}$ for the convergence in law under the law $\mu$
		and $\xrightarrow[]{\;\mu\;}$ for the convergence in probability. 
		\end{theorem-non} 
		
	We are now ready to state the main abstract result relating the convergence of the normalised interarrival times process and the respective induced version.

	\begin{thm} \label{abstract-theorem}
			\leavevmode\\
		Let $A \in \mathscr{B}$ with $\mu(A) > 0$.
		Let $(E_n)_{n\in \mathbb{N}}$ be a sequence of asymptotically rare events.
		Assume that the following properties hold:
		\begin{enumerate}
				\item \label{hyp:1} $\mu(E'_n)\RT_{E_n} \xrightarrow[]{\,\mu_{E'_n}\,} 0$
				\item \label{hyp:2} $\mu_{E_n} \left( \RT_{E_n} < \RT_A\right) \xrightarrow[]{} 0$.
			\end{enumerate}		
			Let $\Phi$ be a random element of $[0,\infty)^{\mathbb{N}}$. Then, we have
		\begin{equation*}
				\label{eq:Bn-prime-convergence}
				\mu_A\big(E'_n\big) \Phi^A_{E'_n} \xRightarrow[]{\;\mu_A\;\,}  \Phi
			\end{equation*}
		if and only if
		\[
		\mu(E_n) \Phi_{E_n} \xRightarrow[]{\;\mu\;\,} \Phi.
		\]
	In particular, $\mu(E'_n) \isEquivTo{\scaleto{+\infty}{4pt}} \mu(E_n)$ for $n$ large enough.
	\end{thm}
			\begin{rem}
			Roughly speaking, the first condition ensures that we go fast enough from $E'_n$ to $E_n$ so that this lag is negligible in the limit. It is similar to the hypothesis made in \cite{Zwe18}. The second condition is necessary to establish the 
			connection between the 
			convergence of  the hitting times point process for the induced system and the corresponding one for the original system. Note that this condition was not needed to establish the connection between the first hitting time for the original system and for the induced one.  It is 
			designed to guarantee 
			that, possibly, we will only miss one cluster and is crucial to make the link between the statistics between the shadowing sets $E'_n$ and our target sets $E_n$.
		\end{rem}
		
		\begin{proof}
			Let us first analyse our shadowing set $E'_n$ and at its measure. Recall that
			\[
			E'_n = \bigcup_{k \geq 0} A\cap \{r_A > k\} \cap T^{-k}E_n
			\]
			Now, by Poincar\'e recurrence theorem, we know that, up to a $\mu$-negligible set in $A$, 
			\begin{multline*}
				E'_n = \bigcup_{k \geq 0} A\cap \{r_A > k\} \cap T^{-k}\left(E_n \cap \{r_A \leq r_{E_n}\}\right) \cup\\
				\cup \bigcup_{k \geq 0} \bigcup_{p \geq 0} A\cap \{r_A > k\} \cap T^{-k}\left(E_n \cap \{r_A  > r_{E_n}\}\right) \cap T^{-p}A.
			\end{multline*}
			
			Now, we will only show that the second term of the union is included in the first one. Let $p>k$ (if not the intersection is empty) and $x \in A\cap \{r_A > k\} \cap T^{-k}\left(E_n \cap \{r_A  > r_{E_n}\}\right) \cap T^{-p}A$. Since $T^kx 
			\in E_n \cap \{r_A  > 
			r_{E_n}\}$, we consider $q = \max \{k < \ell < p \; | \; T^{\ell}x \in E_n\}$.
			Then, by definition of $q$ we have 
			\[
			x \in A \cap \{r_A > q\} \cap T^{-q}(E_n) \subset \bigcup_{k \geq 0} A\cap \{r_A > k\} \cap T^{-k}\left(E_n \cap \{r_A \leq r_{E_n}\}\right).
			\]
			So, 
			\[
			E'_n = \bigcup_{k \geq 0} A\cap \{r_A > k\} \cap T^{-k}\left(E_n \cap \{r_A \leq r_{E_n}\}\right)
			\]
			
			As the terms in the union are pairwise disjoint, we finally get 
			\[
			\mu(E'_n) = \mu\left(E_n \cap \{r_A \leq r_{E_n}\}\right).
			\]
			
			From the second hypothesis, we get
			
			\begin{equation} \label{use_H2}
				\frac{\mu \left( E_n \cap \{r_{E_n} < r_A\} \right)}{\mu(E'_n)} = \frac{\mu(E_n)}{\mu\left(E_n \cap \{r_A \leq r_{E_n}\}\right)}\mu_{E_n} \left( r_{E_n} < r_A\right) \xrightarrow[n\to +\infty]{} 0.
			\end{equation}
			
			Now, we can go on with the proof. We fix $d \geq 1$ and we want to show that we have the convergence for the first $d$ return times. Let's fix some $\varepsilon > 0$. So, we assume the first condition that is to 
			say
			\begin{equation} \label{conv_induced}
				\mu_A(E'_n) \Phi^A_{E'_n} \xRightarrow[]{\;\mu_A\;} \Phi^*.
			\end{equation}
			Since $E'_n \subset A$, we already have the equivalence of \eqref{conv_induced} with
			\begin{equation*} \label{conv_shadow_global}
				\mu(E'_n) \Phi_{E'_n} \xRightarrow[]{\;\mu\;\,} \Phi^*.
			\end{equation*}
			and
			\begin{equation*} \label{conv_shadow_global_starting_A}
				\mu(E'_n) \Phi_{E'_n} \xRightarrow[]{\;\mu_A\;} \Phi^*.
			\end{equation*}
			See \cite[Theorem 11.1 and Proposition 3.1]{Z22} or \cite{FFTV16} for example.
			We write $\Phi^*=\left(\phi^{(1)}, \ldots, \phi^{(n)}, \ldots\right)$. 
			This means that for every $i \in\{1, \ldots, d\}$,
			\begin{equation} \label{hypothesis-conv-implication-1}
			\mu\left(B_j^{\prime}\right) r_{B_j^{\prime}}^{(i)} \underset{j \rightarrow+\infty}{\stackrel{\;\mu, \mu_A\;}{\Longrightarrow}} \phi^{(i)} . 
			\end{equation}
			
			Then, we can consider $t$ large enough so that $\forall i \in \{1,\dots, d\}$, 
			\begin{equation} \label{hypothesis-conv-implication-2}
				\mu \left(\mu(E'_n)r_{E'_n}^{(i)} \geq t\right) \leq \varepsilon.
			\end{equation}
			Now, we only need to show that for every $i \in \{1,\dots, d\}$, we have 
			\begin{equation}
				\label{eq:claim1}
				\mu(E'_n)\left(r^{(i)}_{E_n} - r^{(i)}_{E'_n} \right) \xrightarrow[]{\mu_A} 0.
			\end{equation}
			
			We proceed by (strong) induction. We will prove the property: 
			\begin{equation*}
				\label{eq:induction}
				\mu(E'_n)\left(r^{(i)}_{E_n} - r^{(i)}_{E'_n} \right) \xrightarrow[n \to +\infty]{\mu} 0\quad\text{and}\quad\mu_A \left(r^{(i+1)}_{E_n} < r^{(i+1)}_{E'_n} \right)\xrightarrow[n \to +\infty]{\mu} 0.
			\end{equation*}
			
			As $\mu(E_n) \to 0$, for $n$ large enough we have $\mu\left(r_{E_n} < r_A\right) \leq \varepsilon$. Let $\delta > 0$. We have 
			\begin{align*}
				& \mu \left(\left|r_{E_n} - r_{E'_n}\right| \geq \frac{\delta}{\mu(E'_n)}\right)\\
				& \leq \mu \left(\left\{r_{E_n} - r_{E'_n} \geq \frac{\delta}{\mu(E'_n)}\right\} \cap \left\{r_{E'_n} \leq \frac{t}{\mu(E'_n)}\right\} \cap \{r_A \leq r_{E_n}\} \right) + 2\varepsilon.  \\
				& \leq \sum_{p=1}^{\lfloor t/\mu(E'_n)\rfloor} \mu\left(r_{E'_n} = p, T^{-p} \left(E'_n \cap \left\{r_{E_n} \geq \frac{\delta}{\mu(E'_n)}\right\}\right)\right) + 2\varepsilon \\
				& \leq t\mu_{E'_n} \left( \left\{r_{E_n} \geq \frac{\delta}{\mu(E'_n)}\right\}\right) + 2\varepsilon \\
				& \leq 3\varepsilon \quad \text{for $n$ large enough by hypothesis \eqref{hyp:1}.}
			\end{align*}
			Now, we get that $\mu(E'_n)r_{E_n}$ also converges to $\varphi^{(1)}$ according to $\mu$, so we can find $t'$ such that $\mu \left(\mu(E'_n)r_{E_n} \geq t'\right) \leq \varepsilon$.
			\begin{align*}
				\mu \left(r^{(2)}_{E_n} < r_{E'_n}^{(2)}\right) & \leq  \mu \left( \left\{r_{E'_n} \leq r_{E_n}\right\} \cap \left\{r^{(2)}_{E_n} < r_{E'_n}^{(2)}\right\} \right) + \mu(r_{E'_n} > r_{E_n}) \\
				& \leq \mu \left( \left\{r_{E'_n} \leq r_{E_n}\right\} \cap \left\{r^{(2)}_{E_n} < r_{E'_n}^{(2)}\right\} \cap \{r_{E_n} \leq \frac{t'}{\mu(E'_n)}\} \right) + \varepsilon \\
				&\qquad+ \mu(r_A > r_{E_n}) \\
				& \leq \sum_{p=1}^{\lfloor t'/\mu(E'_n)\rfloor} \mu\left(r_{E_n} = p , \; T^{-p} \left(E_n \cap \left\{r_{E_n} < r_A\right\}\right)\right) + 2\varepsilon \\
				&\leq  t' \frac{\mu\left(E_n \cap \left\{r_{E_n} < r_A\right\}\right)}{\mu(E'_n)} + 2\varepsilon
				\leq 3\varepsilon \quad \text{for $n$ large enough by \eqref{use_H2}}.
			\end{align*}
			
			Now, we consider $i \in \{1,\dots,d-1\}$. The induction step follows the same argument as the base case but we need to consider the convergence to 0 for the first $i$ terms. As $d$ is fixed and finite, it will only add a finite number of $
			\varepsilon$, which will not affect the convergence to $0$. Let $\delta > 0$. We have 
			\begin{align*}
				&\mu\Big( \left|r^{(i+1)}_{E_n} - r^{(i+1)}_{E'_n}\right|\geq \frac{\delta}{\mu(E'_n)}\Big) \\
				& \leq \mu \Bigg(\!\!\left\{r^{(i+1)}_{E_n} - r^{(i+1)}_{E'_n} \geq \frac{\delta}{\mu(E'_n)}\right\}\! \cap\! \left\{r^{(i+1)}_{E'_n} \leq \frac{t}{\mu(E'_n)}\right\} \!\cap\! \bigcap_{k=0}^{i} \big\{r^{(k)}_{E'_n} < r^{(k)}_{E_n}\big\} \!\Bigg) \!+ 
				(d+1)\varepsilon  \\
				& \leq \sum_{p=1}^{\lfloor t/\mu(E'_n)\rfloor} \mu\left(r^{(i+1)}_{E'_n} = p, T^{-p} \left(E'_n \cap \left\{r_{E_n} \geq \frac{\delta}{\mu(E'_n)}\right\}\right)\right) + (d+1)\varepsilon \\
				& \leq t\mu_{E'_n} \left( \left\{r_{E_n} \geq \frac{\delta}{\mu(E'_n)}\right\}\right) + (d+1)\varepsilon \\
				& \leq (d+2)\varepsilon \quad \text{for $j$ large enough by assumption.}
			\end{align*}
			
			Now, we get that $\mu(E'_n)r^{(i+1)}_{E_n}$ also converges to $\varphi^{(i+1)}$ according to $\mu$, so we can find $t'$ such that $\mu \left(\mu(E'_n)r^{(i+1)}_{E_n} \geq t'\right) \leq \varepsilon$.
			\begin{align*}
				&\mu \Big(r^{(i+2)}_{E_n} <r_{E'_n}^{(i+2)}\Big) \\
				& \leq  \mu \left( \bigcap_{k = 1}^{i+1} \left\{r^{(k)}_{E'_n} \leq r^{k}_{E_n}\right\} \cap \left\{r^{(i+2)}_{E_n} < r_{E'_n}^{(i+2)}\right\} \right) + d\varepsilon \\
				&  \leq  \mu \left( \bigcap_{k = 1}^{i+1} \left\{r^{(k)}_{E'_n} \leq r^{(k)}_{E_n}\right\} \cap \left\{r^{(i+2)}_{E_n} < r_{E'_n}^{(i+2)}\right\} \cap \left\{r_{E_n}^{(i+1)} \leq \frac{t'}{\mu(E'_n)}\right\} \right)+ (d+1)\varepsilon \\
				& \leq \sum_{p=1}^{\lfloor t'/\mu(E'_n)\rfloor} \mu\left(r_{E_n}^{(i+1)} = p , \; T^{-p} \left(E_n \cap \left\{r_{E_n} < r_A\right\}\right)\right)+ (d+1)\varepsilon \\
				& \leq t' \frac{\mu\left(E_n \cap \left\{r_{E_n} \leq r_A\right\}\right)}{\mu(E'_n)} + (d+1)\varepsilon\leq (d+2)\varepsilon \quad \text{for $n$ large enough by \eqref{use_H2}}.
			\end{align*}
			This ends the induction step and concludes the proof of \eqref{eq:claim1} and thus the first implication. \\

            \noindent The proof of the reciprocal can be dealt with the same arguments. One only needs to adapt equations \eqref{hypothesis-conv-implication-1} and \eqref{hypothesis-conv-implication-2}, which should now involve $\mu(E_n)\varphi_{E_n}$, instead, since this is the convergence we assume in this case.
		\end{proof}

		\noindent Of course, if $E_n \subset A$, then $E'_n = E_n$ and the result was already known. What is more interseting is when $\mu(E_n \cap A^c) > 0$. We give here two corollaries for shrinking balls around point $\zeta$ that are not included in $A$. They will be 
		useful for the proof of the dichotomy in the following sections. We introduce the notation $\mathcal{O}(\zeta)$ for the orbit of the point $\zeta$, which is to say that $\mathcal{O}(\zeta)=\{\zeta, T(\zeta), T^2(\zeta),\ldots\}$.
		
		\begin{cor} \label{when_hits_the_reference_set}
			Let $A \in \mathscr{B}$ with $\mu(A) > 0$. Let $\zeta \in \mathcal{X}$ be such that $\mathcal{O}(\zeta) \cap \mathring{A}\neq \emptyset$ and $T$ is continuous on the orbit of $\zeta$ until at least $T^{r_A(\zeta)-1}$. Let $(B_n)_{n\in \mathbb{N}}$ be a sequence 
			of asymptotically rare balls shrinking to $\zeta$. Assume that
			\begin{equation*} \label{condition1}
				\mu(B'_n)r_{B_n} \xrightarrow[n\to +\infty]{\mu_{B'_n}} 0.
			\end{equation*}
			Let $\Phi$ be a random element of $[0,\infty)^{\mathbb{N}}$.
			Then, we have
			\begin{equation}
					\label{eq:Bn-prime-convergence}
					\mu_A(B'_n) \Phi^A_{B'_n} \xRightarrow[n\to +\infty]{\mu_A} \Phi
			\end{equation}
			if and only if	
			\[
			\mu(B_n) \Phi_{B_n} \xRightarrow[j\to +\infty]{\mu} \Phi.
			\]
			In particular, $\mu(B'_n) = \mu(B_n)$ for $n$ large enough.
		\end{cor}
		
		\begin{proof}
			The only thing missing is condition $2$ of Theorem \ref{abstract-theorem}. This is immediate since $B_n \cap \{r_{B_n} < r_A\} = \emptyset$ by continuity and because $T^q(B_n) \subset A$ for $n$ large enough. From the formula given in the Theorem, we also 
			have $\mu(B_n) = \mu(B'_n)$ once $n$ was chosen large enough too.
		\end{proof}
		
		\begin{cor} \label{when_does_not_hit}
			Let $A \in \mathscr{B}$ with $\mu(A)> 0$. Let $\zeta \in \mathcal{X}$ be such $\mathcal{O}(\zeta) \cap A = \emptyset$. 
			Let $(B_n)_{n\in \mathbb{N}}$ be a sequence of asymptotically rare balls shrinking to $\zeta$. For all $n \in \mathbb{N}$, we define $Q(B_n) = 
			B_n \cap T^{-s}B_n^{\mathrm{c}}$ if $\zeta$ is a periodic point of prime period $s$ and $Q(B_n) = B_n$, otherwise.  Assume the two following properties :
			\begin{itemize}
				\item[\textup{(}1\textup{)}] 
				$\mu(B'_n)r_{Q(B_n)} \xrightarrow[]{\mu_{B'_n}} 0,$
				\item[\textup{(}2\textup{)}] 
				$\mu_{Q(B_n)} \left( r_{Q(B_n)} < r_A\right) = \mu_{Q(B_n)} \left( r_{B_n} < r_A\right)  \xrightarrow[]{} 0.$
			\end{itemize}
			Let $\Phi$ be a random element of $[0,\infty)^{\mathbb{N}}$.
			Then, we have
			\[
			\mu_A(B'_n) \Phi^A_{B'_n} \xRightarrow[]{\;\mu_A\;\,} \Phi
			\]
			if and only if
			\[
			\mu(Q(B_n)) \Phi_{Q(B_n)} \xRightarrow[]{\;\mu\;\,} \Phi.
			\] 
			In particular, $\mu(B'_n) \isEquivTo{\scaleto{+\infty}{4pt}}  \mu(Q(B_n))$.
		\end{cor}
		\begin{rem}
		The set $Q(B_n)$ is the escape annulus defined in \cite{FFT12} (its relatively simple expression is due to the choice of our target set around a periodic point of period $s$).
		\end{rem}
		\begin{proof}
			This is the application of Theorem \ref{abstract-theorem} applied to the asymptotically rare events $(Q(B_n))_{n\in \mathbb{N}}$. One just need to remark that $Q(B_n)' = B'_n$. The estimate for $\mu(B'_n)$ also comes from the proof of Theorem \ref{abstract-theorem} and especially condition $2$.
		\end{proof}

	\section{Dichotomy for Misiurewicz-Thurston quadratic maps}
	\label{sec:Misiurewicz}
	
	Let $a\in[0,2]$ and for $x\in\left[-1,1\right]$, define $T_a:x\mapsto 1-ax^2$.
	\begin{defn}
		We define $\Mis\subset \left[0,2\right]$ the Misiurewicz-Thurston set of parameters such that $T_a$ admits an absolutely continuous invariant probability measure 
		(acip) $\mu_a$ and the critical point $c=0$ is preperiodic (See \cite{M81}). 
	\end{defn}
	
	\begin{rem}
		For example, the full quadratic map ($a=2$ and $T_2 : x \mapsto 1 - 2x^2$) is a Misiurewicz-Thurston map because $T_2(0) = 1$ and for all $n \geq 2$, $T_2^{n}(0) = 0$.
	\end{rem}
	
	The main advantage of the Misiurewicz-Thurston parameters is that it is possible to build a first return map that is uniformly expanding and Markov. The construction of such a tower can be seen for example in \cite{MS93}. We list here some 
	properties. For $a \in \Mis$, 
	
	\begin{enumerate}
		
		\item \label{Induced_is_Markov} 
		There exists an interval $A_a$  containing the critical point $0$ such that $T_{A_a}$, the first return map to $A_a$, is a Markov map. Furthermore, $\mathcal{O}(0) \backslash \{0\} \cap A_a = \emptyset$ and the boundary of 
		$A_a$ consists of a periodic point $\xi_a$ and its opposite $-\xi_a$ (by symmetry $T_a(\xi_a) = T_a(-\xi_a)$). (See \cite[Lemma V.3.2 on page 364]{MS93}.)
		
		\item \label{Bounded_distortion_outside_A} 
		There exists $K < +\infty$ such that for each $n\in \mathbb{N}$ and each interval $J$ with $T_a^i(J) \cap A_a = \emptyset$ for all $i \in \{0,\dots, n-1\}$, the distortion 
		of ${T_a^n}_{\mkern 1mu \vrule height 2ex\mkern2mu J}$ is uniformly bounded by $K$. (See \cite[Proposition V.3.2 on page 364]{MS93}.) 
		
		\item \label{control_large_return_time} 
		Let $\bigcup_j I_j$ be the domain of $T_{A_a}$ and define $k(j)$ by  ${T_{A_a}}_{\mkern 1mu \vrule height 2ex\mkern2mu I_j} = T_a^{k(j)}$. Then,
		\begin{equation*}
			\sum_{j=1}^{+\infty} k(j)\,\leb(I_j) < +\infty.
		\end{equation*}
		(See \cite[Lemma V.3.3 on page 365]{MS93}.) 
		
		\item \label{Exponential_decay_k(j)} 
		Let $\Lambda_n = A_a \backslash \bigcup_{j: k(j) < n} I_j$. Then the Lebesgue measure of $\Lambda_n$ converges exponentially to $0$. Since $\dd\mu_a/\dd\leb$ is bounded away from $0$ and $+\infty$ on $A_a$,
		the $\mu_a$ measure of $\Lambda_n$ also goes exponentially to $0$. We consider $\eta < 1$ such that $\mu_a(\Lambda_n) \leq C\eta^n$. (See \cite[Lemma V.3.3 on page 365]{MS93}.)	
		
		\item  \label{density_Mis} 
		Write $c_k = T^k(0)$ for every $k \in \mathbb{N}$ and let $\rho_a :=\dd\mu_a/ \dd\leb$. There exist $\psi_0$ a $C^1$-function, $w_0, w_1 < 0$
		and constants \[
		C_k^{(0)} = \frac{\rho_a(0)}{\big|\left(T_a^{k-1}\right)'(c_1)\big|^{\frac{1}{2}}}, \quad \left|C_k^{(1)}\right| \leq \frac{U_a}{\big|\left(T_a^{k-1}\right)'(c_1)\big|^{\frac{3}{2}}}, \; \forall k \geq 1
		\]
		with $U_a \neq 0$, such that $\rho_a$ is supported in $[1-a,a]$ and
		\[
		\rho_a(x) = \psi_0(x)+\rho_a^{(0)}(x)+\rho_a^{(1)}(x)
		\]
		where
		\begin{align*}
		\rho_a^{(0)}(x)	& =   \sum_{k = 1}^{+\infty} \frac{C_k^{(0)}}{\sqrt{|x - c_k|}}\mathds{1}_{\{w_0 < s_{k-1}(x - c_k)< 0\}} \\
	         \rho_a^{(1)}(x) &= \sum_{k = 1}^{+\infty}C_k^{(1)} \sqrt{|x - c_k|}\,\mathds{1}_{\{w_1 < s_{k-1}(x - c_k)< 0\}}
		\end{align*}
		where $s_k = \sgn( (T_a^k)'(c_1)) = \sgn( (T_a^k)'(1))$. \\
		Furthermore, there exist some constants $c > 0$ such that $\rho_a \geq c$. (See \cite{Rue09} or \cite[Section 5 formula (50)]{BS21}.)
		
		\item \label{phi-mixing}
		Defining $\omega = \omega_0 := \{I_j,\; j\in\mathbb{N}\}$ the measurable partition of $A_a$ and by recursion $\omega_{n+1} = \omega \bigvee T^{-1}\omega_n$, we say that $I$ is an $(n+1)$-cylinder if $I\in \omega_n$. 
		Furthermore, we denote $\mathcal{F}_{j,k}$ the $\sigma$-algebra generated by $\omega_j,\dots, \omega_k$. Then, $(A_a, T_{A_a}, \mu_{A_a})$ is exponential $\phi$-mixing, that is to say there exists $C>0$ and 
		$0<\lambda<1$ such that for every $J \in \mathcal{F}_{0,k}$ and $D \in \mathscr{B}([-1,1])$,
		\begin{equation*}
			\label{eq:exponential-phi-mixing}
			\left|\mu_a\left(J \cap T_A^{-(n+k)} D\right)- \mu_a(J)\mu_a(D)\right| \leq C\lambda^n\mu_a(D).  
		\end{equation*}
		(See \cite[Theorem 1 (a) page 5]{AN05} for example.)
	\end{enumerate}

	\begin{rem} \label{Induced_is_Rychlik}
		The induced system $(A_a, T_{A_a}, \mu_{A_a})$ on $A_a$ is also a Rychlik system in the sense of \cite{R83}.	
	\end{rem}
	
	\begin{rem} \label{can_take_small}
		The inducing interval $A_a$ can be chosen as small as we want. One just needs to consider a periodic orbit with at least a point sufficiently close to the critical point. Then, one chooses the point of this orbit that is closest to 
		the critical point and takes its opposite as boundaries of the inducing interval. This can be done to get arbitrarily small diameter. Note that, with this construction, it is guaranteed that the periodic orbit never hits the interior 
		of the inducing set, but rather its boundary, only.
	\end{rem}
	
	\begin{rem}
		We remark that the quadratic family used in \cite{BS21}, from where we took the density formula, was slightly different, so we need to make the necessary adjustments.
	\end{rem}
	
	\begin{cor} \label{leb_equiv_mu}
		If $B$ is a set such that $B \cap \mathcal{O}(0) \backslash \{0\} = \emptyset$, then we have $\mu_{a,B} \asymp \leb_B$, where $\mu \asymp \nu$ means that there are some constants $c, C > 0$ such that $c\nu(A) \leq 
		\mu(A)\leq C\nu(A)$ for every  measurable $A$. In particular, this is the case if we take $B = A_a$.
	\end{cor} 
	
	In the following, as we fix the index $a \in \Mis$, we will drop the indices $a$ in the definitions.	We now state the dichotomy for Misiurewicz-Thurston maps. 
	
	\begin{thm} \label{Dicho_Mis}
		Let $a \in \Mis$ and $([-1,1], T, \mu)$ the corresponding system. Let $\varphi : [0,1] \to \mathbb{R}$ be a distance observable achieving a maximum at $\zeta \in \text{supp}(\mu) = [1-a, 1]$. Let $N_n$ be the REPP associated 
		to $\varphi$ and $u_n$ such that $B_n:= \{\varphi >u_n\}$ satisfies $n\mu(B_n) \xrightarrow[n\to+\infty]{} \tau > 0$. Then, 
		\begin{itemize}
			\item[\textup{(}i\textup{)}] If $\zeta$ is not periodic, $N_n$ converges in distribution to $N$, a homogeneous Poisson Process with intensity $1$.
			\item[\textup{(}ii\textup{)}] If $\zeta$ is periodic of period $p$ and not in the critical orbit, then $N_n$ converges to $N_{\theta, \mathrm{Geo}(\theta)}$ a compound Poisson Process with intensity $\theta = 1 - |(T^p)'(\zeta)|^{-1}$ and 
			multiplicity distribution function $\pi$ given by $\pi(k) = \theta(1-\theta)^{k-1}$ for $k\in \mathbb{N}^*$.
			\item[\textup{(}iii\textup{)}] If $\zeta$ is periodic of period $p$ and belongs to the critical orbit, then $N_n$ converges to $N_{\theta, \mathrm{Geo}(\theta)}$, a compound Poisson Process with intensity $\theta = 1 - |(T^p)'(\zeta)|^{-1/2}$ 
			and multiplicity distribution function $\pi$ given by $\pi(k) = \theta(1-\theta)^{k-1}$ for $k\in \mathbb{N}^*$.
		\end{itemize}
	\end{thm}

	The rest of this section is dedicated to the proof of Theorem \ref{Dicho_Mis}, which will be long and split into several different cases. 
	
	\subsection{Preparatory results and observations}
	
	We start by noting that Theorem \ref{Dicho_Mis} holds for all $\zeta \in A$, by direct application of \cite[Theorem~3]{FFTV16}. Therefore, henceforth, even if not mentioned explicitly, we will always assume that $\zeta\notin A$.
	
	\begin{lem} \label{Arriving_fast_enough}
		Let $\zeta \notin A$. Then, for the sequence $(B_n)_n$ of shrinking balls to $\zeta$, we have
		\begin{equation*}
			\mu_a(B'_n)\,\RT_{B_n} \xrightarrow[n\to + \infty]{\mu_{B'_n}} 0.
		\end{equation*}
	\end{lem}
	
	\begin{proof}
		Consider $\varepsilon > 0$. We have 
		\begin{align*}
			B'_n \cap \{\RT_{B_n} \geq \varepsilon/ \mu(B'_n)\}  &= \bigcup_{k\geq \varepsilon/ \mu(B'_n)} \left(A \cap \{\RT_A > k\} \cap T^{-k}(B_n)\right) \\
			& \subset \bigcup_{j: k(j) \geq \varepsilon/ \mu(B'_n)} I_j \subset \Lambda_{\varepsilon/ \mu(B'_n)}.
		\end{align*}
		Hence,
		\begin{align*}
			\mu_{B'_n}(\RT_{B_n} \geq \varepsilon/ \mu(B'_n)) &\leq \mu(\Lambda_{\varepsilon/ \mu(B'_n)})/\mu(B'_n) \\
			& \leq C\eta^{\varepsilon/ \mu(B'_n)}/\mu(B'_n)  \xrightarrow[n\to + \infty]{} 0 \quad \text{by property \eqref{Exponential_decay_k(j)}}.
		\end{align*}
	\end{proof}
	
	Let $\zeta \notin A$. We need to understand better the structure of the shadow set $B'_n$ of $B_n$, in $A$. Let 
	\begin{equation*}
		\Pre_k = \big\{\alpha \in A : T^k(\alpha) = \zeta \; \text{and}\; T^j(\alpha) \notin A \; \forall j \in \{1,\dots,k\}\big\}\quad\text{and}\; \Pre = \bigcup_{k\geq 1} \Pre_k.
	\end{equation*}  We also write $A = A_- \cup A_+$ where $A_{\pm} = A\cap\mathbb{R}_{\pm}$ the two symmetric parts of $A$. Because the map is symmetric, we have $\Pre_k \cap A_- = \text{sym}(\Pre_k \cap A_+)$.
	
	\begin{lem} \label{rewrite_B'_n}
		For $n$ large enough, there exists a family of intervals $(B'_n(\alpha))_{\alpha \in \Pre}$ included in $A$ such that 
		\[B'_n = \bigcup_{\alpha \in \Pre} B'_n(\alpha) \;\; \text{and the union is disjoint.}\]
		Furthermore, if $\zeta \notin \mathcal{O}(0)$, then for all  $k \geq 1$ and for all $\alpha \in \Pre_k$, we have $T^k(B'_n(\alpha)) = B_n$.
		In particular,
		\[A \cap \{\RT_A > k\} \cap T^{-k}(B_n) = \bigcup_{\alpha \in \Pre_k} B'_n(\alpha).\]
	\end{lem}
	
	\begin{proof}
		Consider first $\alpha \in \Pre$. Thus, there exists $k \in \mathbb{N}$ such that $T^k(\alpha) = \zeta$ and $T^j(\alpha) \notin A $ for all $j \in \{1,\dots,k\}$. If $\alpha \neq 0$, there exists a branch $J_k(\alpha)$ such that $T^k_{|J_k(\alpha)}$ is monotone. We consider its restriction $A_k(\alpha)$ to $A$. The endpoints of $T^k(A_k(\alpha))$ belong to $\mathcal{O}(0) \cup \mathcal{O}(\xi)$ which is finite ($0$ and $\xi$ are pre-periodic and periodic, respectively). Thus, if $\zeta \notin \mathcal{O}(0) \cup \mathcal{O}(\xi)$, as $\zeta \in T^k(A_k(\alpha))$, we can choose $n$ large enough (independently of $k$) so that $B'_n(\alpha) \subset A_k(\alpha)$ is an interval such that $T^k(B'_n(\alpha)) = B_n$. 
		
		Furthermore, for all $j \in \{1,\dots, k\}$, $T^j(B'_n(\alpha)) \cap A = \emptyset$. Indeed we cannot have $T^j(B'_n(\alpha)) \subset A$ because $T^j(\alpha) \notin A$ and $\alpha \in B'_n(\alpha)$. Now, if  $T^j(B'_n(\alpha)) \cap A \neq \emptyset$ and $T^j(B'_n(\alpha))  \not\subset A$, since $B'_n(\alpha)$ is an interval, so is $T^j(B'_n(\alpha))$ and thus it must contain $\xi$ or $-\xi$, but, in this case, since $B_n \cap \mathcal{O}(\xi) = \emptyset$, we have $T^k(B'_n(\alpha)) \neq B_n$ which is a contradiction. Thus, $B'_n(\alpha) \subset B'_n$. 		
		
		If $\alpha = 0$ (in particular $\zeta \in \mathcal{O}(0)$), we have two branches $J_k(0)^-$ and $J_k(0)^+$ symmetric such that $T^k_{|J_k(0)^-}$ and $T^k_{|J_k(0)^+}$ is monotone. We can have the same reasoning as before for each branch independently but since $T(0) = 1$ and $T^2(0) = 1-a$ are the edges of the invariant interval, we will have to consider the restrictions $B'_n(0)^{-}$ and $B'_n(0)^{+}$, but $T^k(B'_n(0)^{\pm}) \subsetneq B_n$ (because only one side is covered). Note that if $\zeta = 1$ or $\zeta= 1-a$ we have $T^k(B'_n(0)^{\pm}) = B_n$ because $\supp(\mu) = [1-a, 1]$. We write $B'_n(0) = B'_n(0)^{-} \cup B'_n(0)^{+}$. This is an interval centred at 0.

		When $\zeta \notin \mathcal{O}(0)$, for all $k\geq 1$, $T^{-k}B_n$ is a union of intervals, $\{\RT_A > k\}$ is also a union of intervals (always having the convention $\RT_A(x) = +\infty$ if $T^nx \notin A$ for all $n \geq 1$) and $A$ is an interval. So $A \cap \{\RT_A > k\} \cap T^{-k}(B_n)$ is a union of intervals. Of course, each component is included in $A$ by construction. We want to show that each component $C$ is equal to $B'_n(\alpha)$ for some $\alpha \in \Pre_k$. Consider a component of $A \cap T^{-k}(B_n)$. We can identify it with a  component $C_{n,k,i}$ of $T^{-k}(B_n)$ intersected with $A$. But $\partial A \cap C_{n,k,i} = \emptyset$ for $n$ large enough such that $B_n \cap \mathcal{O}(\partial A) = \emptyset$ (which is possible since $B_n$ is a ball centred at $\zeta$ and $\zeta \notin \partial A$). Thus $C_{n,k,i} \subset A$. Now, we consider a component of $C_{n,k,i} \cap \{\RT_A > k\}$. Assume that $C_{n,k,i} \cap \{\RT_A > k\} \neq \emptyset$. By contradiction, assume further that $C_{n,k,i} \not\subset \{\RT_A > k\}$. We recall that $C_{n,k,i}$ is still a component of $T^{-k}B_n$, so $T^k(C_{n,k,i}) \subset B_n$. Since $C_{n,k,i} \cap \{\RT_A > k\} \neq \emptyset$ and $C_{n,k,i} \cap \{\RT_A \leq k-1\} = C_{n,k,i} \cap \{\RT_A \leq k\} \neq \emptyset$, there is $1\leq j\leq k - 1$ such that $T^j(C_{n,k,i}) \cap A \neq \emptyset$ and $T^j(C_{n,k,i}) \not\subset A$. But $T^j(C_{n,k,i})$ is an interval meaning that $\partial A \cap T^j(C_{n,k,i}) \neq \emptyset$ which is impossible because $\mathcal{O}(\partial A) \cap B_n = \emptyset$ and $T^{k-j}(T^j(C_{n,k,i})) \subset B_n$, with the previous choice of $n$ large enough. So $C_{n,k,i} \subset \{\RT_A > k\}$. 
		
		We have just shown that, for $n$ large enough, each component $C$ of $A \cap \{\RT_A > k\} \cap T^{-k}(B_n)$ is actually a component of $T^{-k}(B_n)$. But every component of $T^{-k}(B_n)$ contains a preimage $\alpha$ of $\zeta$. The condition $\{\RT_A > k\}$ makes sure that $\alpha \in \Pre_k$. Thus, $C = B'_n(\alpha)$. \\ \\
		Since each $B'_n(\alpha)$ is included in $\{\RT_A >k\}$ for $\alpha \in \Pre_k$, then for every $p \neq k$ and $\alpha \in \Pre_p,\, \alpha' \in \Pre_k$, we must have $B'_n(\alpha) \cap B'_n(\alpha') = \emptyset$. If $\alpha,\alpha' \in \Pre_k$ and $\alpha \neq \alpha'$, by monotony of $T^k$ (since $\alpha$ or $\alpha'$ is different from 0), then $B'_n(\alpha) \cap B'_n(\alpha') = \emptyset$, again. 
	\end{proof}
	
	The idea is now to approximate $B'_n$ by a union of cylinders so that we can use the good mixing properties on cylinders and ultimately obtain the desired convergence in the spirit of \cite{HP14,KY21}, for example.
	
	For that purpose, we will need to distinguish between points $\zeta$ for which $\mathcal{O}(\zeta) \cap \mathring{A}\neq \emptyset$ and such that $\mathcal{O}(\zeta) \cap A = \emptyset$. The case $\mathcal{O}(\zeta) \cap \partial A \neq \emptyset$ will be handled easily in the end (see Remark~\ref{rem:border-case}).

	When $\mathcal{O}(\zeta) \cap A = \emptyset$, for every $\alpha \in \Pre$, we have $\mathcal{O}(\alpha) \backslash \{\alpha\} \cap A = \emptyset$ and therefore $\Pre$ consists of special points where the induced map is not defined, \ie it consists of points where we will have an accumulation of branches with return time to $A$ growing to $+\infty$ (thus an accumulation of 1-cylinders in the sense of the construction of the Young tower).

	\begin{lem} \label{cyl_mixing}
		Let $\mu$ be a $\phi$-mixing measure. Then, there exist positive constants $C$ and $\lambda <1 $ such that for all $n \geq 1$ and all $A = [a_0^{n-1}] \in \omega_n$, 
		\begin{equation*}
			\mu(A) \leq C\mu([a_0])\lambda^n.
		\end{equation*}
	\end{lem}
	
	For the proof, see for example \cite[Lemma 1]{AAG21}, \cite{Aba01}.

	For a set $B\subset A$ and $v \in \mathbb{N}^{*}$, we define 
	\[
	U^+(B, v) := \bigcup_{A \in \omega_{v-1} : A \cap B \neq \emptyset} A\qquad\text{and}\qquad U^-(B, v) := \bigcup_{A \in \omega_{v-1} : A \subset B} A.
	\]
	By definition, $U^+(B,v), U^{-}(B,v) \in \mathcal{F}_{0,v-1}$. They are the approximations of $B$ from above and below by $v$-cylinders.
	
	\subsection{The case where the orbit of  $\zeta$ hits the interior of the inducing set}

	The main goal of this subsection is to prove the following:

	\begin{prop} \label{Dicho_hits}
		Theorem \ref{Dicho_Mis} holds for every $\zeta \in [1-a,1]\setminus A$ such that $\mathcal{O}(\zeta) \cap \mathring{A}\neq \emptyset$.
	\end{prop}
	
	We note that the case studied in this subsection could be covered with standard already available methods, but for the sake of completeness, to illustrate the application of our approach and to pave the way for the following sections, we do it carefully. 
	
	Also, recall that, as observed in Remark~\ref{can_take_small}, by construction of the inducing base, the proposition above does not cover the points in $\mathcal{O}(0)\backslash \{0\}$.
	
	The idea is to use Corollary~\ref{when_hits_the_reference_set}. Since $\mathcal{O}(\zeta) \cap \mathring{A}$, let $q$ be $\min \{n\in \mathbb{N}^* : T^n(\zeta)\in A\}$. 
	We have $B_n := \{\varphi >u_n\} = (\zeta-r_n, \zeta+r_n)$ for some $r_n \to 0$. By continuity, we can assume that $n$ is large enough so that $T^q(B_n) \subset A$. Recall that $B'_n \subset A$ is our shadowing set. The conclusion of Lemma 
	\ref{Arriving_fast_enough} is enough to apply Theorem \ref{when_hits_the_reference_set}. Thus, the convergence  of the REPP for $B_n$ follows from the convergence of the REPP for $B'_n$ under the induced map, which means that the proof of 
	Proposition~\ref{Dicho_hits} is then reduced to the proof of the following result.
	\begin{prop}
		\label{prop:B_n-prime-REPP-convergence}
		Let $\zeta$ be as in Proposition~\ref{Dicho_hits} and consider the sequence $(B_n)_n$ defined as in Theorem~\ref{Dicho_Mis}.  Then the convergence in \eqref{eq:Bn-prime-convergence} holds for the respective shadowing sequence $(B_n')_n$ defined through \eqref{eq:B-prime-definition}.
	\end{prop}
	
	Our strategy to prove this result is to show first that the convergence in \eqref{eq:Bn-prime-convergence} for $B'_n$ can be obtained from establishing the same convergence for an approximating union of cylinders such as $U^+(B'_n,v_n)$ and  $U^{-}(B'_n,v_n)$, for some well chosen sequence $(v_n)_n$ of integers. 	Then, using the fact that the induced map is $\phi$-mixing, we show that a sequence such as  $U^+(B'_n,v_n)$ and  $U^{-}(B'_n,v_n)$ satisfies the dependence conditions, which give the convergence of the REPP associated to them. 	
	
	\begin{lem} \label{approx_cyl_hits}
		Let $(v_n)_n$ be a sequence of positive integres such that $\ln(n) = o(v_n)$. Then, 
		\begin{equation*}
			\mu\left(U^+(B'_n,v_n)\backslash U^-(B'_n,v_n)\right) \leq \rho_n \mu(B'_n),
		\end{equation*}
		with $\rho_n \to 0$. In particular, the convergence of the REPP counting the number of hits to $B'_n$ is equivalent to the convergence of the REPP counting the number of hits to  $U^+(B'_n,v_n)$ (or $U^-(B'_n,v_n)$) and the limits are the same.
	\end{lem}
	
	\begin{proof}
		By Lemma \ref{rewrite_B'_n}, we have
		\[B'_n = \bigcup_{k\geq 1}\; \bigcup_{\alpha \in \Pre_k} B'_n(\alpha)\]
		and the union is disjoint. Thus, we can look at each $\alpha$ independently. For $\alpha \in \Pre_k$ and since $T^q\zeta \in \mathring{A}$, we choose $n$ large enough so that $T^q(B_n) \subset A$. Hence, $B'_n(\alpha) \subset I_{j(\alpha)}$, where $j(\alpha)$ is such that $k(j(\alpha)) = k + q$ (note that $I_{j(\alpha)}$ are also disjoint). Thus, by Lemma \ref{cyl_mixing},
		\[\mu \left(U^+(B'_n(\alpha),v_n)\backslash U^-(B'_n(\alpha),v_n)\right) \leq 2C\mu(I_{j(\alpha)})\lambda^{v_n}.\]
		Then,
		\begin{align*}
			\mu\left(U^+(B'_n,v_n)\backslash U^-(B'_n,v_n)\right) & \leq \sum_{k\geq 1}\sum_{\alpha \in \Pre_k} \left(U^+(B'_n(\alpha),v_n)\backslash U^-(B'_n(\alpha),v_n)\right) \\
			& \leq \sum_{k\geq 1}\sum_{\alpha \in \Pre_k} 2C\mu(I_{j(\alpha)})\lambda^{v_n} \\
			& \leq 2C\lambda^{v_n}\;\; \text{since the $I_{j(\alpha)}$ are disjoint}\\
			& \leq \rho_n \mu(B'_n).
		\end{align*}
		In the last line, we use the fact that $n\mu(B_n) \to \tau >0$ by hypothesis and $\mu(B'_n) = \mu(B_n)$ for $n$ large enough. The condition $\rho_n \to 0$ holds since we imposed $\ln(n) = o(v_n)$.
		
		The fact that the convergence of the respective REPP is equivalent follows from the observation:
		\begin{align*}
			\mu\left(\mathcal N_{U^+(B'_n(\alpha),v_n)}(nJ)-\mathcal N_{B'_n(\alpha)}(nJ)>0\right)&\leq n|J|\mu\left(U^+(B'_n(\alpha),v_n)\setminus B'_n(\alpha)\right)\\
			&\leq n|J|\rho_n\mu(B'_n)\xrightarrow[n\to\infty]{}0,
		\end{align*}
		where we used again the facts that $n\mu(B_n) \to \tau >0$ and $\mu(B'_n) = \mu(B_n)$ for $n$ large.    
	\end{proof}
	
	\begin{proof}[Proof of Proposition~\ref{prop:B_n-prime-REPP-convergence}]
		On account of Lemma~\ref{approx_cyl_hits}, we are left to prove the convergence of the REPP for $U^{+}(B'_n,v_n)$ (or $U^{-}(B'_n,v_n)$). For that end, we use \cite[Theorem~2.A]{FFM18} after showing that the dependence conditions $
		\DD_q(u_n)^*$ and $\DD'_q(u_n)^*$ hold. We remark that when $\zeta$ is a non-periodic point, we will show that $\DD_q(u_n)^*$ and $\DD'_q(u_n)^*$ hold, with $q = 0$ while if $\zeta$ is a periodic point of period $s$, we will show these 
		conditions hold with $q = r := \text{Card}(\mathcal{O}(\zeta) \cap A)$. 
		
		In order to ease the notation we set $U_n^+:=U^+(B'_n,v_n)$. 
		
		\textbf{Condition $\DD_q(u_n)^*$.} We say that $\DD_q(u_n)^*$ holds if for any integers $t, \kappa_1, \dots, \kappa_p$ and any $J = \cup_{i = 2}^p I_j \in \mathcal{R}$ with $\inf J \geq t$,
		\begin{align*}
			& E_n(\kappa_1) \\
			&:= \left|\mu\left(Q_{q,0}^{\kappa_1}(U_n^+) \cap \left(\bigcap_{j=2}^p\mathcal{N}_{U_n^+}(I_j) = \kappa_j\right)\right) - \mu \left(Q_{q,0}^{\kappa_1}(u_n)\right)\mu\left( \bigcap_{j=2}^p\mathcal{N}_{U_n^+}(I_j) = 
			\kappa_j\right)\right|\\ &\leq \gamma(n,t)
		\end{align*}
		where for each $n$, $\gamma(n,t)$ is decreasing in $t$, and $\lim_{n\to\infty}n\gamma(n,t_n)=0$ for some sequence $t_n = o(n)$, where the annulus $Q_{q,0}^{\kappa_1}(U_n^+)$, in the particular cases we are handling here, can be written 
		as:
		\begin{align*}
		& Q_{q,0}^{\kappa_1}(U_n^+)  = \bigcap_{j = 0}^{\kappa_1}T_A^{-jr}(U_n^+) \cap T_A^{-(\kappa_1 + 1)r}((U_n^+)^c) \quad \text{if}\quad q=r>0\\
		& Q_{q,0}^{\kappa_1}(U_n^+)=U_n^+ \quad \text{if} \quad q=0. 
		\end{align*}
		When $q>0$, note that $E_n(\kappa) \leq 2\mu \left(Q_{q,0}^{\kappa}(U_n^+)\right) \leq 2C\theta(1-\theta)^{\kappa}\mu(U_n^+)$, since $\mu$ is regular on the set considered (in fact, we have $1 - \theta = |(T^s)'(\zeta)|^{-1} = |
		(T_A^r)'(T^q(\zeta))|^{-1}$). So we choose $\kappa(n)$ so that $n (1 - \theta)^{\kappa(n)} \to 0$. Hence, we only have to consider $\kappa_1 < \kappa(n)$. Observe that 
		\[
		Q_{q,0}^{\kappa_1}(U_n^+) \in 	\mathcal{F}_{0,(\kappa_1+1)r +v_n} \subset \mathcal{F}_{0,\kappa(n)r+v_n}.
		\]
		Now, by \eqref{phi-mixing}, for some $0<\lambda < 1$ and for $\kappa_1 \leq \kappa(n)$ we have,
		\begin{align*}
			E_n(\kappa_1) \leq C\lambda^{t - r\kappa(n) - v_n}.
		\end{align*}
		Hence, we can take $\gamma(n,t) := \max \left\{C\lambda^{t - r\kappa(n) - v_n}, 2C\theta(1-\theta)^{\kappa(n)}\right\}$. When $q=0$, we only need to consider  $\gamma(n,t) := C\lambda^{t - v_n}$.
		We consider the most complicated of the cases, in which we have 
		\[
		n\gamma(n,t_n) \leq n2C\theta(1-\theta)^{\kappa(n)} + nC\lambda^{t - r\kappa_(n) - v_n}.
		\]
		Recall that, $v_n$ needs to be such that $\ln(n) = o(v_n)$, $\kappa(n)$ such that $n(1 - \theta)^{\kappa(n)}\to 0$. So it is possible to find $t_n = o(n)$ and appropriate $v_n, \kappa(n)$ so that $n\gamma(n,t_n) \to 0$.

		\textbf{Condition $\DD_q'(u_n)$.} For some fixed $q\in\N_0$, consider the sequence $(t_n)_{n\in\N}$, given by condition  $\DD(u_n)^*$ and let $(\ell_n)_{n\in\N}$ be another sequence of integers such that
		\begin{equation}
			\label{eq:kn-sequence}
			\ell_n\to\infty\quad \mbox{and}\quad  \frac{n}{\ell_n} t_n = o(n).
		\end{equation}
		
		We say that condition $\DD'_q(u_n)^*$
		holds if there exists a sequence $(k_n)_{n\in\N}$ satisfying \eqref{eq:kn-sequence} and such that
		\begin{equation*}
			\label{eq:D'rho-un}
			\lim_{n\rightarrow\infty}\,n\sum_{j=q+1}^{\ell_n-1}\mu_A\left( Q_{q,0}^{0}(U_n^+)\cap T_A^{-j}\left(U_n^+\right)
			\right)=0.
		\end{equation*}
		Let 
		\[
		R_n:=\inf\{\RT_{U_n^+}^A(x)\colon x\in Q_{q,0}^{0}(U_n^+)\}.
		\]
		Using again that the induced map is $\phi$-mixing with exponential tails of rate $0<\lambda<1$, we have
		\begin{align*}
		& n\sum_{j = 1}^{\ell_n - 1} \mu_A\left( Q_{q,0}^{0}(U_n^+)\cap T_A^{-j}\left(U_n^+\right)\right) \\
		& \leq n\ell_n \mu_A(Q_{q,0}^{0}(U_n^+))\mu_A(U_n^+) + n\mu_A(Q_{q,0}^{0}(U_n^+)) \sum_{j = R_n}^{+\infty} \lambda^j
		\end{align*}     
		Recalling that $\lim_{n\to\infty}n\mu(B_n)=\tau\geq0$ and observing that, since $T_A(U_n^+)$ is an interval, $\lim_{n\to\infty}R_n=\infty$ either by continuity when $q=0$ and by the Hartman-Grobman Theorem when $q>0$, then the term on 
		the right of the last displayed equation vanishes as $n\to\infty$, which proves $\DD'_q(u_n)^*$.     
		As we said, condition $\DD_q(u_n)^*$ and $\DD'_q(u_n)^*$ are sufficient to show the convergence of the REPP for $U_n^+$ and thus for $B'_n$ (by Lemma~\ref{approx_cyl_hits}) and hence for $B_n$ (by 
		Corollary~\ref{when_hits_the_reference_set}), concluding the proof in this case. Furthermore, we saw that if $\zeta$ is periodic of prime period $s$, we have $\theta = 1 -  |(T^s)'(\zeta)|^{-1}$.    
	\end{proof}
	
	\subsection{The case where the orbit of $\zeta$ does not hit the inducing set} 
	
	\noindent The main goal of this subsection is to prove the following:
	
	\begin{prop} \label{Dicho_does_not_hit}
		Theorem \ref{Dicho_Mis} holds for every $\zeta \in [1-a,1]$ such that $\mathcal{O}(\zeta) \cap A = \emptyset$.
	\end{prop}
	
	When studying REPP in the presence of clustering created by observables maximised at periodic points, we usually observe a limiting compound Poisson process (see \cite{FFT13}, for example), which could be described as having two components: the first is the time occurrences of the clusters, which is ruled by a homogeneous Poisson process, and the second is a Geometric multiplicity distribution, which describes the number of visits to $B_n$ during the same cluster. We observe that an entrance in the annuli $Q(B_n)$ marks to the last hit to $B_n$ within a cluster and, therefore, the point process of entrances in $Q(B_n)$ gives us the time occurrences of clusters.

	One of the main difficulties in this case, where the orbit of $\zeta$ does not hit the inducing set, is that the induced map may miss some of the intra cluster hits to $B_n$, since the orbits may return to $B_n$ without going through the base of the induced map. Hence, we split the analysis by considering first the cluster positions and later, in Section~\ref{reconstruction_cluster},  we reconstruct the point process of hits to $B_n$ from the point process of hits to $Q(B_n)$.

	\begin{prop} \label{convergence_for_Q}
		Let $N'_n$ be the REPP for $Q(B_n)$ renormalized by $\mu(B_n)$ that is to say
		\begin{equation*}
			N'_n(J) = \sum_{i\, \in \,\mu(B_n)^{-1}\!J\,\cap\, \mathbb{N}} \mathds{1}_{Q(B_n)}(T^ix). 
		\end{equation*}
		Then, 
		\begin{equation*}
			N'_n \xrightarrow[n \to +\infty]{\mu} N_{\theta},
		\end{equation*}
		where $N_{\theta}$ a standard Poisson process of intensity $\theta = \lim_{n\to +\infty} \mu(Q(B_n))/\mu(B_n)$.
	\end{prop}
	
	\begin{rem}
		If $\zeta$ is non-periodic, $Q(B_n) = B_n$, by construction, meaning that $N_n'=N_n$, in this case.
	\end{rem}
	
	\begin{proof}
		Again, the first condition of Corollary~\ref{when_does_not_hit} follows from  Lemma \ref{Arriving_fast_enough}. We need to prove the second condition to be able to use Corollary~\ref{when_does_not_hit}. \\
		
		Since $B_n$ are balls centred at $\zeta$ and $\mathcal{O}(\zeta) \cap A = \emptyset$, there exists $\delta > 0$ (that may depend on $\zeta$ but not on $n$) such that for $n$ large enough, there exist $k_n\in \mathbb{N}$ with $\leb( T^{k_n}Q(B_n)) \geq \delta$ and 
		$T^{j}(B_n) \cap A = \emptyset$ for $j\in \{0,\dots,k_n-1\}$ (if $\zeta$ is a periodic point or a pre-periodic point, the assertion follows from the application of Grobman-Hartman Theorem; if $\zeta$ is not periodic nor pre-periodic, we can consider an inducing set $A'$ 
		such that $d(\mathcal{O}(\zeta), A') \geq \delta' > 0$ and take $\delta = \delta'$).
		Since $B_n$ is an interval, $T^{k_n}_{|B_n}$ has bounded distortion, given by a constant $K$ that does not depend on $n$ (by \ref{Bounded_distortion_outside_A}). Thus,
		\begin{align*}
			\leb \left(Q(B_n) \cap \{\RT_{B_n} \leq \RT_A\}\right) & \leq K \frac{\leb \left( T^{k_n}\left(Q(B_n)\cap \{\RT_{B_n} \leq \RT_A\}\right)\right)}{\leb(T^{k_n}Q(B_n))}\leb(Q(B_n)) \\
			\leb_{Q(B_n)}\left(  Q(B_n)\cap \{\RT_{B_n} \leq \RT_A\}  \right) &\leq K\delta^{-1}\leb(\{\RT_{B_n}\leq \RT_A\}).
		\end{align*}
		Since $\mu(\RT_{B_n} \leq \RT_A) \xrightarrow[n\to +\infty]{} 0$ and $\mu$ and $\leb$ are equivalent, we have 
		\[
		\leb_{Q(B_n)}\left(  Q(B_n)\cap \{\RT_{B_n} \leq \RT_A\}  \right)\xrightarrow[n\to +\infty]{} 0.
		\] 
		It just remains to show that it implies $\mu_{Q(B_n)}\left(  Q(B_n)\cap \{\RT_{B_n} \leq \RT_A\}  \right)\xrightarrow[n\to +\infty]{} 0$. If $\zeta \notin \mathcal{O}(0)$, then $
		\mu \asymp \leb$, it is immediate. \\
		If $\zeta \in \mathcal{O}(0)$, we use the fact that we know the form of the density and the singularities are exactly of the form $1/\sqrt{|x - \zeta|}$ and thus we also have
		\[
		\mu_{Q(B_n)}\left(  Q(B_n)\cap \{\RT_{B_n} \leq \RT_A\}  \right)\xrightarrow[n\to +\infty]{} 0.
		\] 
		The proof is finished.
		\end{proof}

		So, we can apply Corollary~ \ref{when_does_not_hit} and the convergence of the REPP counting entrances in $B'_n$ for the induced map to obtain the convergence of the REPP counting entrances in $Q(B_n)$. After, we will only need to rebuild the compound 
		process from the REPP counting entrances in $Q(B_n)$. 
		In this case, we just want to show the convergence to a standard Poisson process for the shadowing set $B'_n$ under the induced transformation $T_A$. We recall that $T_A$ is known to be Markov 
		for parameters $a \in \Mis$. However, it is not possible to immediately prove the conditions $\DD_0(u_n)^*$ and $\DD'_0(u_n)$ using the standard proof because $B'_n$ is not a ball around some point $\zeta' \in A$. However, due to its definition, it still has a form that 		we can characterise. Indeed, it consists of the union (at most countable and not necessarily disjoint) of intervals around the preimage of $\zeta$ in $A$, whose orbit does not hit $A$ before arriving at $\zeta$. These intervals are not centred on the preimages but 	
		almost (the difference is only due to the fact that the derivative is not constant but since it is continuous the difference is small).

		\begin{lem} \label{comparison_1cyl}
			Let $\zeta$ be such that  $d(\mathcal{O}(\zeta), A) > \gamma$, for some $\gamma > 0$, and $B_n$ a sequence of shrinking balls to $\zeta$. Then, for every $k\in \mathbb{N}, \; \alpha \in \Pre_k$, we have
			\begin{align*}
				\mu(B'_n(\alpha)) \geq c\mu \left(U^+(B'_n(\alpha),1)\right) \\
				\mu \left(U^-(B'_n(\alpha),1)\right) \geq c\mu(B'_n(\alpha)),
			\end{align*}
			for some $c>0$.
			\end{lem}

		\begin{proof}
			Note that $B'_n(\alpha)$ is an interval and thus $U^+(B'_n(\alpha), 1) \backslash U^-(B'_n(\alpha),1)$ consists of at most two intervals which are located at the extremities of $B_n'(\alpha)$. Set 
			\[
			m := \min \{p\geq 1: T^p\left(U^+(B'_n(\alpha), 1)\right)\cap A \neq \emptyset\}
			\] (note that $m > k$ for $n$ large enough by definition and Lemma \ref{rewrite_B'_n}). Since, $T_{|T(B'_n(\alpha))}^{m-1}$ is monotone (as it does not cover $0\in A$), 
			$T^m\left(U^+(B'_n(\alpha), 1) \backslash U^-(B'_n(\alpha),1)\right)$ is still composed of at most two intervals that lie at the extremities of $T^m\left(U^+(B'_n(\alpha),1)\right)$. Since $U^{+}(B'_n(\alpha),1)$ is composed by 1-cylinders, 
			$T^m\left(U^+(B'_n(\alpha), 1\right) \cap A \neq \emptyset$, we have $A \subset T^m\left(U^+(B'_n(\alpha), 1\right)$. But since $T^{m-k}\zeta \in T^m(B'_n(\alpha))$, then we must have that the interval between $T^{m-k}\zeta$ and the 
			closest element of $\partial A$ is contained in $T^m\left(B'_n(\alpha)\backslash U^+(B'_n(\alpha),1)\right)$. Thus, $\left|T^p(B'_n(\alpha))\right|\geq \gamma$. \\
			Then, again by bounded distortion (Property~\eqref{Bounded_distortion_outside_A}) 
			applied to $T^{m-1}_{|T(U^+(B'_n(\alpha),1))}$, we have
			\begin{align} \label{univ_distortion_comp_1cyl}
				|T(U^+(B'_n(\alpha),1))| &\leq K|T(B'_n(\alpha))|\frac{|T^m(U^+(B'_n(\alpha),1))|}{|T^m(B'_n(\alpha))|}\leq \frac{2K}{\gamma} |T(B'_n(\alpha))|. 
			\end{align}
			We have to analyse two different cases. If $\zeta \cap \mathcal{O}(0) = \emptyset$, for $n$ large enough so that $B_n \cap \mathcal{O}(0) = \emptyset$, 
			we have $|T(U^+(B'_n(\alpha),1))| \geq c' \mu\left(T(U^+(B'_n(\alpha),1))\right) \geq cc'\mu\left(U^+(B'_n(\alpha),1)\right)$, using that $\mu \asymp \leb$ outside the critical orbit. Now, since
			 $|T(B'_n(\alpha))| \leq C|B'_n(\alpha)| \leq CC'\mu(B'_n(\alpha))$, 
			we have $\mu(B'_n(\alpha)) \geq  c \mu\left(U^+(B'_n(\alpha),1)\right)$ for $c > 0$  not depending on $k$. 
			
			In the special case where $\alpha = 0$, we do not have that $|T(U^+(B'_n(0),1))| \geq c' \mu\left(T(U^+(B'_n(0),1))\right)$, anymore. However, by definition of the Misiurewicz map, we have $|T(U^+(B'_n(0), 1))| = a(|U^+(B'_n(0),1)|/2)^2$ 
			and identically $|T(B'_n(0))| = a(|B'_n(0)|/2)^2$. Thus, by \eqref{univ_distortion_comp_1cyl},
			\begin{align*}
				a(|U^+(B'_n(0),1)|/2)^2 &\leq Ca(|B'_n(0)|/2)^2 \\
				|U^+(B'_n(0),1)| &\leq C'|B'_n(0)|.
			\end{align*}
			Now, since both lie inside $A$ and $\mu \asymp \leb$ on $A$, then for a certain $c>0$,  we have $c\mu(U^+(B'_n(0),1)) \leq \mu(B'_n(0))$.
		\end{proof}
		
		\begin{lem}
			Let $(B'_n)$ be as above. Then for every sequence $(v_n)$ diverging to $+\infty$, we have 
			\[
			\mu(U^+(B'_n, v_n) \backslash U^-(B'_n, v_n)) \leq C\lambda^{v_n} \mu(B'_n).
			\]
		\end{lem}
		
		\begin{proof}
			Let $N$ be such that, for all $n>N$,  all the conditions of the previous lemmas are satisfied.  Consider also that  $n$ is large enough so that $T^p(B_n) \subset B_N$. We have that 
			\[
			B'_n = \bigcup_{k \geq 1} \bigcup_{\alpha \in \Pre_k} B'_n(\alpha)
			\]
			and this union is disjoint. Hence
			\[ 
			\bigcup_{k \geq 1} \bigcup_{\alpha \in \\Pre_k} U^-(B'_n(\alpha),v_n) \subset U^{-}(B'_n, v_n) \subset U^{+}(B'_n, v_n) \subset  \bigcup_{k \geq 1} \bigcup_{\alpha \in \Pre_k} U^+(B'_n(\alpha),v_n)
			\]
			leading to 
			\[
			\mu(U^+(B'_n, v_n) \backslash U^-(B'_n, v_n)) \leq \sum_{k\geq 1}\sum_{\alpha\in \Pre_k} \mu(U^+(B'_n(\alpha), v_n) \backslash U^-(B'_n(\alpha), v_n)).
			\]
			Hence, we can treat each $B'_n(\alpha)$ independently. So fix $k\geq 1$ and $\alpha \in \Pre_k$. Since $B'_n(\alpha)$ is an interval, for every $q \geq 1$, there are at most two cylinders of $\omega_{q-1}$ such that $A\cap B'_n(\alpha) 
			\neq \emptyset$ and $A \not\subset B'_n$.  Hence, using the previous lemma
			\begin{align*}
				\mu(U^+(B'_n(\alpha),1) \backslash U^-(B'_n(\alpha), 1)) \leq (c^{-1} - c)\mu(B'_n(\alpha)).
			\end{align*}
			But $U^+(B'_n(\alpha),1) \backslash U^-(B'_n(\alpha), 1)$ consists of at most two disjoint cylinders $C_1, C_2 \in \omega_0$. Thus, by Lemma \ref{cyl_mixing}, 
			\[
			\mu(U^+(B'_n(\alpha),v_n) \backslash U^-(B'_n(\alpha), v_n)) \leq C\lambda^{v_n}\mu(U^+(B'_n(\alpha),1) \backslash U^-(B'_n(\alpha), 1)).
			\]
			This gives
			\[
			\mu(U^+(B'_n(\alpha),v_n) \backslash U^-(B'_n(\alpha), v_n)) \leq C\lambda^{v_n}\mu(B'_n(\alpha)).
			\]
			Now, we can sum for every $k\geq 1$ and $\alpha \in \Pre_k$ to get
			\[ 
			\mu(U^+(B'_n, v_n) \backslash U^-(B'_n, v_n)) \leq C\lambda^{v_n}\mu(B'_n).
			\]
		\end{proof}
		
		Similarly to what we did earlier (see Lemma~\ref{approx_cyl_hits}), by the last lemma, it is equivalent to have the convergence of the REPP counting the number of hits to $U^+(B'_n,v_n)$, $U^-(B'_n, v_n)$ or $B'_n$, as long as $(v_n)$ is 
		chosen as a diverging sequence. By definition, $U^+(B'_n, v_n)$ and $U^-(B'_n, v_n)$ are in $\mathcal{F}_{0,v_n-1}$. \\
		
		Now, we will show the convergence to a standard Poisson process of the REPP counting the number of hits to  $U^+(B'_n,v_n)$, which implies the convergence of the REPP counting the number of hits to $B'_n$. For that purpose, we show 
		that conditions $\DD_0(u_n)^*$ and $\DD'_0(u_n)$ are satisfied for the induced map.
		
		\begin{lem}
			Let $\tau > 0$ and $(u_n)$ a sequence of thresholds such that $n\mu_A(B'_n) \xrightarrow[n\to +\infty]{} \tau$. Then, conditions  $\DD_0(u_n)^*$ and $\DD'_0(u_n)$ hold. This means that the REPP $N^A_n$ converges in distribution to 
			$N$ a standard homogeneous Poisson Process with intensity 1.
		\end{lem}
		
		\begin{proof}
		We recall that, once conditions  $\DD_0(u_n)^*$ and $\DD'_0(u_n)$ are checked, then the conclusion for the REPP comes from \cite[Theorem~2.A]{FFM18}. 
		We saw that the induced map is Rychlik. Thus, we know that it is exponential $\phi$-mixing (Theorem 1-a) in \cite{AN05}). We will check the conditions for $U^+(B'_n,v_n)$, the reasonning is the same if we consider $U^-(B'_n,v_n)$, 
		instead. We have
			\begin{align*}
				& n\sum_{j = 1}^{\ell_n - 1} \mu_A \left(U^+(B'_n,v_n) \cap T_A^{-j} \left(U^+(B'_n,v_n)\right) \right)\\
				& \leq n \sum_{j = v_n + 1}^{\ell_n - 1} \mu_A \left(U^+(B'_n,1) \cap T_A^{-j} \left(U^+(B'_n,v_n)\right) \right) \\ & \quad\quad + n\sum_{j = 1}^{v_n} \mu_A \left(U^+(B'_n,j) \cap T_A^{-j} \left(U^+(B'_n,v_n)\right) \right) \\
				& \leq n \sum_{j = v_n + 1}^{\ell_n - 1} \mu_A\left(U^+(B'_n,1)\right) \mu_A\left(U^+(B'_n,v_n)\right) \\
				& \quad\quad + n \sum_{j = v_n + 1}^{\ell_n - 1} \phi(j - 1)\mu_A\left(U^+(B'_n,1)\right) \\ &  \quad \quad + n\sum_{j = 1}^{v_n} \mu_A \left(U^+(B'_n,j) \cap T_A^{-j} \left(U^+(B'_n,v_n)\right) \right) \\
				& \leq \frac{n^2 \mu_A\left(U^+(B'_n,1)\right) \mu_A\left(U^+(B'_n,v_n)\right) }{k_n} + n\mu_A\left(U^+(B'_n,1)\right) \sum_{j = v_n}^{+\infty} \phi(j) \\ & \quad \quad + n\sum_{j = 1}^{v_n} \mu_A \left(U^+(B'_n,j) \cap T_A^{-j} 
				\left(U^+(B'_n,v_n)\right) \right).
			\end{align*}
			
			Since $k_n, v_n \xrightarrow[n\to+\infty]{} \infty$ and $(A, T_A, \mu_A)$ is summable $\phi$-mixing (in fact, exponential $\phi$-mixing), the first two terms go to $0$ when $n$ goes to $+\infty$. We are left with estimating the last term. For 
			that purpose, we will use bounded distortion, again. Take $1 \leq j \leq v_n$. Since $U^+(B'_n, j)$ is a union of $j$ cylinders, we write $U^+(B_n, j) = \bigcup_{E \in \mathcal{A}^j : E\subset U^+(B_n, j)} E$. Since the interiors of the
			cylinders are disjoint and $\mu_A$ does not charge any mass point, we have in particular 
			\[\mu_A \left( U^+(B_n, j) \right)  = \sum_{E \in \omega_j : E\subset U^+(B_n, j)} \mu_A(E).\]
			But, by bounded distortion there is a constant $K$ independent of $j$ (depends on $A$ but $A$ is fixed at the beginning), such that for each cylinder $E \in \omega_j$ and since our map is Markovian (which implies that $T_A^j : E \to A$ 
			is onto), we have 
			\[\mu_A\left(E \cap T_A^{-j}(U^+(B_n, v_n))\right) \leq K\mu_A\left(E\right) \mu_A\left(U^+(B_n, v_n)\right).\]
			
			Hence,
			\begin{align*}
				& n\sum_{j = 1}^{v_n} \mu_A \left(U^+(B'_n,j) \cap T_A^{-j} \left(U^+(B'_n,v_n)\right) \right)\\
				& = n\sum_{j = 1}^{v_n}\sum_{E \in \omega_j : E\subset U^+(B_n, j)} \mu_A \left(E \cap T_A^{-j} \left(U^+(B'_n,v_n)\right) \right) \\
				& \leq n \sum_{j = 1}^{v_n}\sum_{E \in \omega_j : E\subset U^+(B_n, j)} K\mu_A(E)\mu_A \left(U^+(B_n, v_n)\right) \\
				& \leq K n  \sum_{j = 1}^{v_n} \mu_A\left(U^+(B_n,j)\right) \mu_A \left(U^+(B_n, v_n)\right) \\
				& \leq KC^2 v_n/n \xrightarrow[n \to +\infty]{} 0.
			\end{align*}
		\end{proof}
		
		\begin{lem}
			If $\zeta$ is a periodic point of period $p$ and $\zeta \notin \mathcal{O}(0)$, then 
			\[
			\theta := \lim_{n\to +\infty} \mu(Q(B_n))/\mu(B_n) = 1 - |(T^p)'(\zeta)|^{-1}.
			\]
			If $\zeta$ is a periodic point of period $p$ and $\zeta \in \mathcal{O}(0)$, 
			\[
			\theta = 1 - |(T^p)'(\zeta)|^{-1/2}.
			\]
		\end{lem}
		
		\begin{proof}
		The case when $\zeta \notin \mathcal{O}(0)$ is immediate. Indeed, for $n$ large enough such that $B_n \cap \mathcal{O}(0) = \emptyset$, $\mu_{B_n} \asymp \leb_{B_n}$
		and since $\mu$ is quite regular, then $\theta = 1 - |(T^p)'(\zeta)|^{-1}$.  
			
		If $\zeta \in \mathcal{O}(0)$ and is periodic we  can use the formula given by Proposition \ref{density_Mis} to get the result.
		\end{proof}

	\begin{rem}
		\label{rem:border-case}
		In the proof of Proposition \ref{Dicho_hits} and \ref{Dicho_does_not_hit}, we could not treat special points in $\mathcal{O}(\xi)$ and for which $\xi$ and its symmetric are the borders of $A$. But with remark \ref{can_take_small}, we can choose 
		a smaller interval and thus these points can now be studied with Proposition \ref{Dicho_does_not_hit}. 
	\end{rem}
	
	\subsection{Reconstruction of the clusters} \label{reconstruction_cluster}
	
	In the previous section, we have seen that clusters appear scattered in the time line according to an homogeneous  Poisson process of intensity $\theta := \lim \mu(Q(B_n))/\mu(B_n)$. If $\zeta$ is not a periodic point, by definition, $Q(B_n) = B_n$ 
	and there is nothing more to prove. However, when $\zeta$ is a periodic point, the topological cluster is not seen by the induced map and we need to reconstruct it. As we will see, the nice structure of our shrinking balls $B_n$ and the local 
	dynamics ruled by Hartman-Grobman theorem allow us to rebuild the clusters. 
	
	\begin{rem}
	We remark that the reconstruction procedure is very general and it will work for balls as long as the density is regular enough without any further assumptions.
	\end{rem}  
	
 From this point forward, let us fix $\zeta$ as a periodic point of prime period $p$ and $(B_n)_n$ a sequence of balls shrinking to $\zeta$ as earlier. We introduce some notation following \cite{FFT13}, for example, in order to study the clusters:
	\begin{enumerate}
		\item let $Q_0(B_n) := Q(B_n) = B_n \cap T^{-p}(B_n^c)$ be the outer annuli and for every $k \in \mathbb{N}$, define the higher order annuli as: $Q_{k+1}(B_n) := T^{-p}(Q_{k}(B_n)) \cap B_n$.
		
		\item let $U_0(B_n) := B_n$ and for every $k\in \mathbb{N}$, set $U_{k+1}(B_n) = T^{-p}(U_k(B_n)) \cap B_n$.
	\end{enumerate}
	
	We recall next a useful result in order to establish the convergence of point processes (see for example \cite[Theorem 14.16]{Kal21}).	
	\begin{prop}\label{Kall_conv_PP}
		In order to have the convergence of the point processes $N_n$ to $N$,  it is enough to check that, for  all $I_1, \dots, I_q \subset \mathbb{R}$ with $I_j = [a_j, b_j)$ and $N(\partial I_j) = 0$ a.s, we have 
		\[
		\left(N_n(I_1), \dots, N_n(I_q)\right) \xRightarrow[n\to +\infty]{} (N(I_1), \dots, N(I_q)).
		\] 
	\end{prop}
	Define an adjusted version of the first hitting time to $A$ by
	\[
	h_A(x) := \min\{n\geq 0 : T^n(x) \in A\} 
	\]
	and for every $\ell \in \mathbb{N}$, the corresponding $\ell$-th hitting time to $A$: 
	\[
	h^{(\ell)}_A(x): = \min \{n\in \mathbb{N}: \text{Card}(A\cap \{x, \dots, T^nx\}) = \ell\}.
	\]
	Note that if $x\notin A$, then $h^{(\ell)}(x) = \RT^{(\ell)}(x)$ and if $x \in A$, $h^{(1)}(x) = 0$ and $h^{(\ell+1)}(x) = \RT^{(\ell)}(x)$. We introduce $h^{(\ell)}$ for technical reasons related with the forthcoming definition of the return time processes 
	which will add a mass at $0$. 
	
	Proposition \ref{convergence_for_Q} gives us the convergence: 
	\[N'_n = \sum_{\ell \geq 1} \delta_{\left\{\mu(B_n)h_{Q(B_n)}^{(\ell)}\right\}} \xRightarrow[n\to + \infty]{\mu} N_{\theta},\]
	with $N_{\theta}$ denoting an homogeneous Poisson process of intensity $\theta$. 
		
 	When an orbit enters $Q(B_n)$ it determines the ending of cluster. It will be useful to consider also the beginning of a cluster and, for that purpose, we introduce  the entrance set:
 	\[
	E(B_n) := T^{-p}B_n \cap B^c_n
	\]
	which  marks the beginning of a cluster (after $p$ steps). $Q(B_n)$ has the advantage of being inside $B_n$ but studying hits to $E(B_n)$ helps in relating the point processes of cluster locations and of hits to the target sets. 
	Instead of studying the point process of cluster locations by considering entrances to $Q(B_n)$, we consider the point process counting hits to $E(B_n)$. Namely, let 
	\[
	N''_n := \sum_{\ell \geq 1} \delta_{\left\{\mu(B_n)h_{Q(B_n)}^{(\ell)}\right\}}.
	\]
	
	\begin{lem} \label{from_Q_to_E}
		We have 
		\begin{equation*}
			N''_n \xRightarrow[n\to + \infty]{\mu_{E(B_n)}} N_{\theta} + \delta_0,
		\end{equation*} 
		with $N_{\theta}$ denoting the homogeneous Poisson process with intensity $\theta$.
	\end{lem}
	\begin{proof}
		We use Proposition~\ref{Kall_conv_PP} to obtain the convergence of $N''_n$ under $\mu$. Let $I_1,\dots, I_n$ be such that $N_{\theta}(\partial I_j) = 0$ a.s for every $j \in \{1,\dots,q\}$. We can choose $n$ large enough so that $\mu(\RT_{Q(B_n)} \leq p) \leq \varepsilon$. Observe that on the complement of the set $\{\RT_{Q(B_n)} \leq p\}$, every return to $Q(B_n)$ is preceded by a return to $E(B_n)$. We also consider $\delta > 0$ such that $\mu(N'_n([a_j-\delta, a_j + \delta]) \geq 1), \mu(N'_n([b_j-\delta, b_j + \delta]) \geq 1) \leq \varepsilon$ for every $j \in \{1,\dots,q\}$. Then we have
		\begin{align*}
			& \mu \left(\left(N''_n(I_1), \dots, N''_n(I_n)\right) \neq (N'_n(I_1), \dots, N'_n(I_n))\right) \leq \sum_{j = 1}^q \mu(N''_n(I_j) \neq N'_n(I_j)) \\
			& \leq \sum_{j =1}^q \sum_{\ell = 0}^{k-1} \mu\left(\mu(B_n)\RT^{(\ell)}_{E(B_n)} \leq a_j \leq a_j + \delta \leq \mu(B_n)\RT^{(\ell)}_{E(B_n)}\right)\\
			& \qquad +\sum_{j =1}^q \sum_{\ell = 0}^{k-1} \mu\left(\mu(B_n)\RT^{(\ell)}_{E(B_n)} \leq b_j \leq b_j + \delta \leq \mu(B_n)\RT^{(\ell)}_{E(B_n)}\right) + (2q+1)\varepsilon \\
			& \leq \sum_{j = 1}^q \sum_{\ell = 0}^{k-1} \mu\left(\RT_{Q(B_n)}^{(\ell)} - \RT_{E(B_n)}^{(\ell)} \geq 2\delta/\mu(B_n)\right) +(2q+1)\varepsilon \\
			& \leq \sum_{j = 1}^q \sum_{\ell = 0}^{k-1} \sum_{p = 0}^{\max I_j/\mu(B_n)}\mu\left(\RT_{E(B_n)}^{(\ell)} = p,\; T^{-p}(E(B_n) \cap \{\RT_{Q(B_n)} \geq \delta/\mu(B_n)\})\right) +(2q+1)\varepsilon\\
			& \leq qk\frac{\max I_j}{\mu(B_n)}\mu\left(E_n\cap \{ \RT_{Q(B_n)} \geq \delta/\mu(B_n)\}\right) + (2q+1)\varepsilon\\
			& \leq (2q+2)\varepsilon \quad \text{for $n$ large enough}.
		\end{align*}

		\noindent The convergence in law of $N''_n$ to $N_{\theta}$ under $\mu$ implies the convergence in law of $N''_n$ under $\mu_{E(B_n)}$ to $N_{\theta} + \delta_0$  (\cite[Theorem 3.1]{Zwe16} or \cite[Theorem 1]{HLV07}).
	\end{proof}
	
	In order to compare the measure of the successive annuli  $Q_k(B_n)$ and the $U_k(B_n)$ and the  measure of the respective preimages in $E(B_n)$, we need the following lemma.
	
	\begin{lem} \label{comp_meas_E_n_B_n}
		For every sequence $A_n \in \mathcal{F} \cap B_n$, 
		\begin{equation*}
			T^p_{\#}\mu_{E(B_n)}(A_n) \sim \mu_{B_n}(A_n).
		\end{equation*}
	\end{lem} 
	
	\begin{proof}
		We have
		\begin{align*}
			T^{p}_{\#} \mu_{E(B_n)}(A_n) & = \mu_{E(B_n)}(T^{-p}A_n) = \frac{1}{\mu(E(B_n))} \left(\mu(A_n) - \mu(B_n \cap T^{-p}(A_n))\right)\\
			& =  \frac{1}{\mu(E(B_n))} \left(\mu(A_n \cap B_n) - \mu(U_1(B_n) \cap T^{-p}(A_n))\right)
		\end{align*}
		
		\noindent Now, by change of variables (since $T^p : U_1(B_n) \to B_n$ is one to one and onto), we have
		\begin{align*}
			\mu(A_n\cap B_n) = \int_{B_n} \mathds{1}_{A_n} \rho\,\dd \leb = \int_{U_1(B_n)} \mathds{1}_{A_n} \circ T^p \rho \circ T^p |(T^p)'| \,\dd\leb.
		\end{align*}
		On the other hand, we have
		\begin{align*}
			\mu(U_1(B_n)\cap T^{-p}A_n) = \int_{U_1(B_n)} \mathds{1}_{A_n} \circ T^p \rho \,\dd\leb.
		\end{align*}
		
		Since, we would like to compare both measures, we need to compare $\rho \circ T^p|(T^p)'|$ and $\rho$ on $U_1(B_n).$ Let us fix a small $\varepsilon > 0$. Since $T^p$ is a polynomial, $|(T^p)'|$ is continuous and for $n$ large enough, $|(T^p)'(x)| \in [|(T^p)'(\zeta)| - \varepsilon, |(T^p)'(\zeta)| + \varepsilon]$ for $x \in B_n$. 
		
		Now, we need to consider two cases. First, if $\zeta \notin \mathcal{O}(0)$, by \eqref{density_Mis}, $\rho$ is continuous at $\zeta$ and thus provided $n$ is large enough, $\rho(x) \in [\rho(\zeta) - \varepsilon, \rho(\zeta) - \varepsilon]$ for $x \in B_n$. Since $T^p(U_1(B_n)) \subset B_n$ by construction, $\rho\circ T^p \in [\rho(\zeta) - \varepsilon, \rho(\zeta) + \varepsilon]$ for $x \in U_1(B_n)$. Thus, 
		
		\begin{align*}
			\frac{|(T^p)'(\zeta)|-\varepsilon}{1+\varepsilon} \leq \frac{\mu(A_n\cap B_n)}{\mu(U_1(B_n)\cap T^{-p}A_n)} \leq \frac{|(T^p)'(\zeta)| + \varepsilon}{1 - \varepsilon}.
		\end{align*}
		
		\noindent When, $\zeta \in \mathcal{O}(0) \backslash \{0\}$, $\rho$ has a singularity at $\zeta$, but we can still compare $\rho$ and $\rho \circ T^p$. Indeed,  using \eqref{density_Mis} 
		
		\begin{align*}
			\rho \circ T^p(x) &= \psi_0\circ T^p(x) + \sum_{k = 1}^{+\infty} \frac{C_k^{(0)}}{\sqrt{|T^px - c_k|}}\mathds{1}_{\{w_0 < s_{k-1}(T^px - c_k)< 0\}} \\
			&\qquad + \sum_{k = 1}^{+\infty}C_k^{(1)} \sqrt{|T^px - c_k|}\mathds{1}_{\{w_0 < s_{k-1}(T^px - c_k)< 0\}}.
		\end{align*} 
		
		The only problem is when $c_k = \zeta$ in the central term (all the other are continuous at $\zeta$ and thus equal to a constant up to $\varepsilon$). We have
		
		\begin{align*}
			& \sum_{k = 1, c_k = \zeta}^{+\infty} \frac{C_k^{(0)}}{\sqrt{|T^px - \zeta|}}\mathds{1}_{\{w_0 < s_{k-1}(x - c_k)< 0\}} \\
			& = \frac{C_1}{\sqrt{|T^px - \zeta|}}\mathds{1}_{\{w_0 < (T^px - \zeta)< 0\}} + \frac{C_2}{\sqrt{|T^px - \zeta|}}\mathds{1}_{\{w_0 < -(T^px - \zeta)< 0\}}\\
			& = \frac{C_1}{\sqrt{|T^px - T^p\zeta|}}\mathds{1}_{\{w_0 < (T^px - \zeta)< 0\}} + \frac{C_2}{\sqrt{|T^px - T^p\zeta|}}\mathds{1}_{\{w_0 < -(T^px - \zeta)< 0\}} \\
			& =  \frac{C_1}{\sqrt{|(T^p)'(c)|}\sqrt{|x - \zeta|}}\mathds{1}_{\{w_0 < (T^px - \zeta)< 0\}} + \frac{C_2}{\sqrt{|(T^p)'(c)|}\sqrt{|x - \zeta|}}\mathds{1}_{\{w_0 < -(T^px - \zeta)< 0\}} \\
			& = \frac{1}{\sqrt{|(T^p)'(c)|}}\left(\frac{C_1}{\sqrt{|x - \zeta|}}\mathds{1}_{\{w_0 < (T^px - \zeta)< 0\}} + \frac{C_2}{\sqrt{|x - \zeta|}}\mathds{1}_{\{w_0 < -(T^px - \zeta)< 0\}}\right)
		\end{align*}
		
		\noindent Thus, provided $n$ is large enough, 
		\begin{align*}
			\frac{\sqrt{|(T^p)'(\zeta)|} - C\varepsilon}{1+\varepsilon} \leq \frac{\mu(A_n\cap B_n)}{\mu(U_1(B_n)\cap T^{-p}A_n)} \leq \frac{\sqrt{|(T^p)'(\zeta)|} + C\varepsilon}{1+\varepsilon}.
		\end{align*}
		
		Hence, returning to the expression of $T^p_{\#}\mu_{E(B_n)}$, we get 
		\begin{align*}
			T^{p}_{\#} \mu_{E(B_n)}(A_n) & = \mu_{E(B_n)}(T^{-p}A_n) = \frac{1}{\mu(E(B_n))} \left(\mu(A_n) - \mu(B_n \cap T^{-p}(A_n))\right)\\
			& =  \frac{1}{\mu(E(B_n))} \left(\mu(A_n \cap B_n) - \mu(U_1(B_n) \cap T^{-p}(A_n))\right) \\
			& \sim \frac{\theta}{\mu(E(B_n))} \mu(A_n \cap B_n) \\
			& \sim \frac{1}{\mu(B_n)}\mu(A_n \cap B_n)\\
			& \sim \mu_{B_n}(A_n).
		\end{align*}
	\end{proof}
	
Let $K^{(\ell)}_n(x)$ be the unique $k \in \mathbb{N}$ such that $T^p \circ T^{h^{(\ell)}_{E(B_n)}(x)}(x) \in Q_k(B_n)$, which is to say that $K^{(\ell)}_n$ is the size of the $\ell$-th cluster. Then, we define the point process:
\[
\widehat{N}_n := \sum_{\ell\geq 1} K^{(\ell)}_n \delta_{\left\{\mu(B_n)h_{E_n}^{(\ell)}\right\}}.
\]
\begin{lem} \label{from_PPP_to_CPP}
We have that
\[
\widehat{N}_n \xRightarrow[]{\;\mu_{E(B_n)}\;} N_{\theta, \mathrm{Geo}(\theta)} + X_1\delta_0
\]
where $N_{\theta, \mathrm{Geo}(\theta)}$ is a compound Poisson process with intensity $\theta$ and geometric multiplicity distribution law of parameter $\theta$, while $X_1\sim \mathrm{Geo}(\theta)$ is independent of $N_{\theta, 
\mathrm{\mathrm{Geo}}(\theta)}$.
\end{lem}
	
\begin{rem}
We observe that since we are studying returns now instead of hits, we obtain a slightly different limiting process, in particular, with a point mass at $0$.
\end{rem}

	\begin{proof}
		First, note that $\mu(E_n) = \mu(T^{-p}B_n \cap B_n^c) = \mu(T^{-p}B_n) - \mu(T^{-p}B_n \cap B_n) = \mu(B_n) - \mu(B_n \backslash Q(B_n)) = \mu(Q(B_n))$. 
		
		Furthermore, we have for every $\ell,k \in \mathbb{N}$,
		\begin{align*}
			\mu_{E(B_n)}\left(K^{(\ell)}_n \geq k\right) & = \mu_{E(B_n)} \left(T^p \circ T^{\RT_{E(B_n)}^{(\ell)}} \in U_{k-1}(B_n)\right) \\
			& = \mu_{E(B_n)} \left(T^{-p} \left(U_{k-1}(B_n)\right)\right) \\
			& = \frac{1}{\mu(Q(B_n))} \mu \left(T^{-p}\left(U_{k-1}(B_n)\right) \cap B_n^c\right) \\
			& = \frac{\mu\left(T^{-p} \left(U_{k-1}(B_n)\right)\right) - \mu \left(T^{-p}\left(U_{k-1}(B_n)\right) \cap B_n\right)}{\mu(Q(B_n))} \\
			& = \frac{\mu \left(U_{k-1}(B_n)\right) - \mu \left(U_{k}(B_n)\right)}{\mu(Q(B_n))} \\
			& = \frac{\mu(Q_{k-1}(B_n))}{\mu(Q(B_n))} \xrightarrow[n\to +\infty]{} (1-\theta)^{k}.
		\end{align*}
		So, $K^{(\ell)}_n \xRightarrow[n\to + \infty]{\mu_{E_n}} K^{(\ell)} \sim \mathrm{Geo}(\theta)$ for every $\ell \geq 1$. 
		
		Moreover, starting from $E(B_n)$, we have for $\ell < \ell'$ and $k,k' \in \mathbb{N}$,
		\begin{align*}
			& \mu_{E(B_n)} \left(\left(K^{(\ell)}, K^{(\ell')}\right) = (k,k') \right) \\
			& = \mu_{E(B_n)}\left(K_1 \circ T^{\ell}_{E(B_n)} = k, \; K_{\ell'-\ell+1} \circ T^{\ell}_{E(B_n)} = k'\right) \\
			& = \mu_{E(B_n)}\left(K^{(1)}_n = k,\; K^{(\ell - \ell'+1)} = k'\right).
		\end{align*}
		Thus, it is enough to look at the independence between $K^{(1)}$ and $K^{(\ell)}$.
		
		\begin{align*}
			&\mathbb{P}(K^{(1)}_n \geq k, K^{(\ell)}_n = k')  = \mu_{E(B_n)} \left(K^{(1)}_n \geq k,\; K^{(\ell)}_n = k'\right) \\
			&\qquad = \mu_{E(B_n)} \left(T^{-p} \left(U_{k-1}(B_n)\right) \cap T^{-\ell + 1}_{E(B_n)}\left(T^{-p}(Q_{k'-1}(B_n))\right)\right) \\
			&\qquad = \mu_{E(B_n)} \left(T^{-p} \left(U_{k-1}(B_n) \cap T^{-\ell + 1}_{E(B_n)}\left(Q_{k'-1}(B_n)\right)\right)\right) \\ 
			&\qquad = \frac{1 }{\mu(E(B_n))} \Biggl( \mu \left( T^{-p} \left(U_{k-1}(B_n) \cap T^{-\ell + 1}_{E(B_n)}\left(T^{-p}(Q_{k'-1}(B_n))\right)\right)\right) \\&\qquad\qquad\qquad -
			\mu \left( B_n \cap T^{-p} \left(U_{k-1}(B_n) \cap T^{-\ell + 1}_{E(B_n)}\left(T^{-p}(Q_{k'-1}(B_n))\right)\right) \right)  \Biggl) \\
			&\qquad =  \frac{1}{\mu(E(B_n))} \Biggl(\mu \left( U_{k-1}(B_n) \cap T^{-\ell + 1}_{E(B_n)}\left(T^{-p}(Q_{k'-1}(B_n))\right)\right) \\ &\qquad\qquad\qquad-
			\mu \left( T^{-\ell + 1}_{E(B_n)}\left(T^{-p}(Q_{k'-1}(B_n))\right) \cap U_{k}(B_n) \right) \Biggr) \\
			&\qquad = \frac{1}{\mu(E(B_n))}\mu \left( T^{-\ell + 1}_{E(B_n)}\left(T^{-p}(Q_{k'-1}(B_n))\right) \cap Q_{k-1}(B_n) \right) \\
			&\qquad \sim \frac{(1-\theta)^k}{\mu(E(B_n))}\mu\left(T^{-\ell + 1}_{E(B_n)}\left(T^{-p}\left(Q_{k'-1}(B_n)\right)\right) \cap Q(B_n)\right).
		\end{align*}
		
		\noindent Now, we have
		\begin{align*}
			&\mu\left(T^{-\ell + 1}_{E(B_n)} \left(T^{-p}\left(Q_{k'-1}(B_n)\right)\right) \cap Q(B_n)\right) = \mu \left(T^{-p}\left(T^{-\ell + 1}_{E(B_n)}\left(T^{-p}\left(Q_{k'-1}(B_n)\right)\right) \cap Q(B_n)\right)  \right) \\
			&\qquad = \mu \left(T^{-p}\left(T^{-\ell + 1}_{E(B_n)}\left(T^{-p}\left(Q_{k'-1}(B_n)\right)\right) \cap Q(B_n)\right) \cap E(B_n) \right) \\&\qquad\qquad\qquad + \mu \left(T^{-p}\left(T^{-\ell + 1}_{E(B_n)}\left(T^{-p}\left(Q_{k'-1}(B_n)\right)\right) \cap Q(B_n)\right) \cap B_n \right) \\
			&\qquad = \mu \left(E(B_n) \cap T^{-p} Q(B_n) \cap T^{-p}\left(T^{-\ell + 1}_{E(B_n)}\left(T^{-p}\left(Q_{k'-1}(B_n)\right)\right) \right)\right) \\&\qquad\qquad\qquad + \mu\left(T^{-\ell + 1}_{E(B_n)} \left(T^{-p}\left(Q_{k'-1}(B_n)\right)\right) \cap Q_1(B_n)\right)\\
			&\qquad = \sum_{j = 0}^{+\infty}  \mu \left(E(B_n) \cap T^{-p} Q_j(B_n) \cap T^{-p(j+1)}\left(T^{-\ell + 1}_{E(B_n)}\left(T^{-p}\left(Q_{k'-1}(B_n)\right)\right) \right)\right) \; \text{by immediate recursion}.\\
			&\qquad = \sum_{j = 0}^{+\infty}  \mu \left(E(B_n) \cap T^{-p} Q_j(B_n) \cap T^{-\ell + 1}_{E(B_n)}\left(T^{-p}\left(Q_{k'-1}(B_n)\right) \right)\right) \\
			&\qquad = \mu \left(E(B_n) \cap T^{-\ell + 1}_{E(B_n)}\left(T^{-p}\left(Q_{k'-1}(B_n)\right) \right)\right) 
		\end{align*}
because $\left(E(B_n) \cap T^{-p}(Q_j(B_n))\right)_{j\in \mathbb{N}}$ are disjoint.		
		\noindent Thus,
		\begin{align*}
			\mu_{E(B_n)} \left(K^{(1)}_n \geq k, \; K^{(\ell)}_n = k'\right) &\sim (1-\theta)^k \mu_{E(B_n)} \left(T^{-\ell + 1}_{E(B_n)}\left(T^{-p}\left(Q_{k'-1}(B_n)\right) \right)\right) \\
			& \sim (1 - \theta)^k \mu_{E(B_n)} \left(T^{-p}Q_{k'-1}(B_n)\right) \\
			& \sim \mathbb{P}(K^{(1)} \geq k)\mathbb{P}(K^{(\ell)} = k').
		\end{align*}
		
		\noindent Since, we already know that $K^{(1)}$ and $K^{(\ell)}$ converge, $(K^{(1)},K^{(\ell)})$ is tight and the only possible limit is the product of two independent random variables with a geometric distribution, $\mathrm{Geo}(\theta)$. 
		
		 To obtain the independence between $K^{(\ell)}$ the sizes of the cluster and the successive return times, the same proof applies. 
		  Finally, the independence and law of the successive return times is given by Lemma \ref{from_Q_to_E}. 
		  
		  Hence, we have proved the limit
		\begin{equation*}
			\widehat{N}_n \xRightarrow[n \to +\infty]{\mu_{E(B_n)}} N_{\theta, \mathrm{Geo}(\theta)} + X_1\delta_0,
		\end{equation*}
		where $N_{\theta, \mathrm{Geo}(\theta)}$ is a compound Poisson process with intensity $\theta$ and a geometric multiplicity distribution  of parameter $\theta$ and $X_1 = K^{(1)}\sim \mathrm{Geo}(\theta)$ is independent of $N_{\theta, 
		\mathrm{Geo}(\theta)}$.
	\end{proof}
	
	Consider now the REPP counting the entrances in $B_n$: 
	\[
	N_n : = \sum_{\ell \geq 1} \delta_{\left\{\mu(B_n)h_{B_n}^{(\ell)}\right\}}.
	\]
	
	\begin{thm} \label{from_E_to_B}
		We have
		\begin{equation*}
			N_n \xRightarrow[n\to +\infty]{\mu} N_{\theta, \mathrm{Geo}(\theta)},
		\end{equation*}
		where $N_{\theta, \mathrm{Geo}(\theta)}$ is a compound Poisson process with intensity $\theta$ and a geometric multiplicity distribution of parameter $\theta$.
	\end{thm}

	\begin{proof}
		By Lemma \ref{from_PPP_to_CPP}, $\widehat{N}_n$ converges to $N_{\theta, \mathrm{Geo}(\theta)} + X_1\delta_0$ under $\mu_{E(B_n)}$.  We first show that $N_n$ converges to $N_{\theta, \mathrm{Geo}(\theta)} + X_1\delta_0$ under $\mu_{E(B_n)}$. Let $\varepsilon > 0$.  
		By Lemma~\ref{from_Q_to_E}, we can choose $k \in \mathbb{N}$ such that $\mu_{E(B_n)}(N''_n([0, \max I_j])\geq k) \leq \varepsilon$ and since $N_{\theta, \mathrm{Geo}(\theta)}(\partial I_j) = 0$ a.s for all $j \in \{1,\dots,q\}$ we can consider $\delta >0$ such that $\mu_{E(B_n)}(N''_n([a_j-\delta, a_j + \delta]) \geq 1), \mu_{E(B_n)}(N''_n([b_j-\delta, b_j + \delta]) \geq 1) \leq \varepsilon$, for every $j \in \{1,\dots,q\}$. 
		Now,
		\begin{align*}
			& \mu_{E(B_n)}\left(\left(N_n(I_1), \dots, N_n(I_q)\right) \neq (\widehat{N}_n(I_1), \dots, \widehat{N}_n(I_q)) \right) \leq \sum_{j = 1}^q \mu_{E_n}\left(N_n(I_j) \neq \widehat{N}_n(I_j)\right) \\
			& \leq \sum_{j = 1}^{q} \sum_{\ell = 0}^{k-1}\mu_{E(B_n)}\left(\mu(B_n)\RT^{(\ell)}_{E(B_n)} < a_j \leq \mu(B_n)\RT^{(\ell)}_{Q(B_n)}\right) \\&\qquad\qquad + \sum_{j = 1}^{q} \sum_{\ell = 0}^{k-1}\mu_{E(B_n)}\left(\mu(B_n)\RT^{(\ell)}_{E(B_n)} < b_j \leq \mu(B_n)\RT^{(\ell)}_{Q(B_n)}\right) + \varepsilon \\
			& \leq \sum_{j = 1}^q \sum_{\ell = 0}^{k-1} \mu_{E(B_n)}\left(\mu(B_n)\RT_{Q(B_n)} \geq \delta \right) + (2q+1)\varepsilon \\
			& \leq qk\mu_{E(B_n)}\left(\mu(B_n)\RT_{Q(B_n)}\geq \delta\right) +(2q+1)\varepsilon \\
			& \leq (2q+2)\varepsilon \quad \text{for $n$ large enough}.
		\end{align*}
		This means,
		\[N_n \xRightarrow[n\to + \infty]{\mu_{E(B_n)}} N_{\theta, \mathrm{Geo}(\theta)} + X_1\delta_0.\]
		
		\noindent Noting that $N_n$ under $T^p_{\#}\mu_{E(B_n)}$ corresponds to $N_n \circ T^p$ under $\mu_{E(B_n)}$ and by construction we do not miss any cluster, then,  since $p$ is fixed, the normalisation by $\mu(B_n)$ makes the difference disappear asymptotically. Thus, 
		
		\[N_n \xRightarrow[n\to + \infty]{T^p_{\#}\mu_{E(B_n)}} N_{\theta, \mathrm{Geo}(\theta)} + X_1\delta_0.\]
		By Lemma \ref{comp_meas_E_n_B_n}, we have $T^p_{\#}\mu_{E(B_n)}(A_n) \sim \mu_{B_n}(A_n)$ for every sequence $A_n \in \mathcal{F} \cap B_n$. Thus, 
		
		\begin{align*}
			&\mu_{B_n}\left((N_n(I_1), \dots, N_n(I_q)) = (k_1, \dots,k_q)\right)\\
			&  \sim T^p_{\#}\mu_{E(B_n)} \left((N_n(I_1), \dots, N_n(I_q)) = (k_1, \dots,k_q)\right)\\
			& \xrightarrow[n\to +\infty]{} \mathbb{P}\left((N_{\theta, \mathrm{Geo}(\theta)}(B_1), \dots, N_{\theta, \mathrm{Geo}(\theta)}(B_q)) = (k_1, \dots,k_q)\right),
		\end{align*}
		
		This proves the convergence of $N_n$ to $N_{\theta, \mathrm{Geo}(\theta)} + X_1\delta_0$ under $\mu_{B_n}$. Again by equivalence between hitting and return time processes (\cite[Theorem 3.1]{Zwe16} or \cite[Theorem 1]{HLV07}), it is also equivalent to the convergence of $N_n$ under $\mu$ and we have 
		\[
		N_n \xRightarrow[]{\;\mu\;\,} N_{\theta, \mathrm{Geo}(\theta)}.
		\]
	\end{proof}

	\section{Dichotomy for doubly intermittent maps}
	\label{sec:doubly-intermittent}
	
	\noindent The purpose of this section is to establish the dichotomy for doubly intermittent full branch maps having neutral points at both ends of the intervals. We recall first the properties of these maps. See \cite{CLM22} for more details.
	
	\begin{defn}
	Let $I = [-1,1]$, $I_- = [-1,0]$ and $I_+ = [0,1]$.\\
	We assume:
	\begin{enumerate}[label = (A\arabic*)]
		\item $T : I \to I$ is full branch, the restricitions $T_{\pm} : I_{\pm} \to I$ are orientation preserving $C^2$ diffeomorphisms and the only fixed points are the endpoints of $I$. 
		\item There exists $\ell_1,\ell_2 \geq0,\; k_1,k_2,a_1,a_2,b_1,b_2 > 0$ such that
		\begin{enumerate}[label = (\roman*)]
			\item if $\ell_1,\ell_2 \neq 0$ and $k_1,k_2 \neq 1$, then
			\begin{equation*}
				Tx =  \begin{cases} 
					x+ b_1(1+x)^{(1+\ell_1)} & \text{in} \; U_{-1}, \\
					1 - a_1|x|^{k_1} & \text{in} \; U_{0-}, \\
					-1 + a_2x^{k_2} & \text{in} \; U_{0+}, \\
					x - b_2(1-x)^{1+\ell_2} & \text{in}\; U_{-1},
				\end{cases}
			\end{equation*}
			\noindent where $U_{0-} := (-\iota, 0]$, $U_{0+} = [0,\iota)$ and $U_{\pm 1} := T(U_{0\pm})$.
			\item If $\ell_1 = 0$ and/or $\ell_2 = 0$, 
			\begin{equation*}
				T_{|U_{\pm 1}}x := \pm 1 + (1+b_1)(x+1) \mp \xi(x),
			\end{equation*}
			where $\xi$ is $C^2$. \\
			If $k_1 = 1$ and/or $k_2 = 1$, we only ask $T'(0_-) = a_1 > 1$ and/or $T'(0_+) =a_2 >1$ and $T$ is monotone in the corresponding meighbourhood.
		\end{enumerate}  
	\end{enumerate}
	
	\noindent We define 
	\begin{equation*}
		\Delta_0^- := T^{-1}(0,1) \cap I_{-} \quad \text{and} \quad \Delta_{0}^+ := T^{-1}(-1,0) \cap I_{+},
	\end{equation*}
	and by recursion 
	\begin{equation*}
		\Delta_n^- := T^{-1}(\Delta_{n-1}^-) \cap I_{-} \quad \text{and} \quad \Delta_{n}^+ := T^{-1}(\Delta_{n-1}^+) \cap I_{+}.
	\end{equation*}
	By construction and hypothesis on $T$, $\{\Delta_{n}^{\pm}\}_{n\geq 0}$ is a partition of $I_{\pm}$. Furthermore, we define
	\begin{equation*}
		\delta_n^- := T^{-1}(\Delta_n^+) \cap \Delta_0^{-} \quad \text{and} \quad \delta_n^+ := T^{-1}(\Delta_n^-) \cap \Delta_0^{+}.
	\end{equation*}
	This time, $\{\delta_n^{\pm}\}_{n\geq 1}$ is a partition of $\Delta_{0}^{\pm}$ and we have $T^n(\delta_n^{\pm}) = \Delta_0^{\mp}$. Let $n_{\pm} := \min \{n\; |\; \delta_{n}^{\pm} \subset U_{0\pm}\}$. We also assume\\
	\noindent (A2) There exists $\lambda >1$ such that for all $1\leq n\leq n_{\pm}$, for all $x \in \delta_{n}^{\pm}$, we have $(T^n)'(x) > \lambda$. 
	\end{defn}

	\noindent We denote $\widehat{\mathfrak{F}} := \{T : I \to I,\; T \;\text{satisfies (A0)-(A2)}\}$. \\
	\noindent Let $\beta := \beta_1 \vee \beta_2$ with $\beta_1 := k_1\ell_1$ and $\beta_2:= k_2\ell_2$. We define $\mathfrak{F} := \{T \in \widehat{\mathfrak{F}}: \beta < 1\}$.

	\begin{prop} \cite[Theorem B]{CLM22}\label{existence_acip}
		For all $T \in \mathfrak{F}$, $T$ admits an ergodic invariant probability $\mu$ equivalent to $\leb$ and bounded away from $0$ and $+\infty$ on $\Delta_{0}^{-} \cup \Delta_{0}^+$. In particular, we have $\leb_{|\Delta_{0}^{-} \cup \Delta_{0}^+} \asymp \mu_{\Delta_{0}^{-} \cup \Delta_{0}^+}$.
	\end{prop}

	\begin{prop}
		For $T \in \mathfrak{F}$, $T_{\Delta_0^{\pm}}$ is a first return Gibbs-Markov and Rychlik map.
	\end{prop}
	 
	\begin{prop} \label{return_times}
		If $\beta = 0$, then $\mu_{\Delta_{0}^{-}}(\RT_{\Delta_{0}^{-}}> t)$ decays exponentially as $t \to +\infty$. If $\beta >0$,  there exists $C > 0$ such that
		\begin{equation*}
			\mu_{\Delta_{0}^{-}}(\RT_{\Delta_{0}^{-}}> t) \leq Ct^{-1/\beta}.
		\end{equation*}
	\end{prop}
	\begin{rem}
	 	The advantage of this symmetrical construction on $\Delta_0^-$ and $\Delta_0^+$ is that we can consider the appropriate induction set depending on the relative position of $\zeta$.
	\end{rem}
	\begin{rem}
		For the notations, we will use $T_{-} : I_{-} \to I$ the left branch of $T$ and $T_{+} : I_+ \to I$ the right-branch. Both of them are continuous, increasing, one to one and onto.
	\end{rem}
	
	\noindent Now, we are able to formulate the dichotomy theorem for doubly intermittent full branch maps.
	
	\begin{thm} \label{Dichotomy_double_intermittent}
		Let $T \in \mathfrak{F}$ and $\mu$ be its associated acip. Let $\varphi : I \to \mathbb{R}$ be a distance observable achieving a maximum at $\zeta \in I$. Let $N_n$ be the REPP associated to $\varphi$ and $(u_n)_{n\in\N}$ a sequence such that $B_n := \{\varphi > u_n\}$ satisfies $n\mu(B_n) \xrightarrow[n\to +\infty]{} \tau > 0$. Then, 
		\begin{enumerate}[label = (\roman*)]
			\item If $\zeta \in (-1,1)$ and $\zeta$ not periodic, $N_n$ converges in distribution to $N$ an homogeneous Poisson Process with intensity 1.
			\item If $\zeta \in (-1,1)$ is periodic of period $p$, then $N_n$ converges to $N_{\theta, \mathrm{Geo}(\theta)}$ a compound Poisson Process with intensity $\theta = 1 - |(T^p)'(\zeta)|^{-1}$ and multiplicity distribution function $\pi$ given by $\pi(k) = \theta(1-\theta)^{k-1}$ for $k\geq 1$.
			\item If $\ell_1,\ell_2 \neq 0$ and $\zeta \in \{-1,1\}$, $N_n$ does not converges but $N_{Q(B_n)}$ converges to an homogeneous Poisson process of intensity $1$. We still have the convergence of the Hitting and Return Time Statistics but with another renomarlisation.
		\end{enumerate}
	\end{thm}
	
	\begin{rem}
		When $\zeta \in \Delta_{0}^- \cup \Delta_{0}^+$, we can direclty use the induction results so there is nothing to prove. However, we do not need to make this distinction as points in $\Delta_{0}^- \cup \Delta_{0}^+$ are carried the same way in the proof.
	\end{rem}

	\begin{lem} \label{condition_1_thm}
		Let $\zeta \notin \Delta_{0}^-$. Then, for a sequence $(B_n)_n$ of shrinking balls to $\zeta$, we have 
		\begin{equation*}
			\mu(B'_n)\,\RT_{B_n} \xrightarrow[]{\;\mu_{B'_n}\;\,} 0.
		\end{equation*}
	\end{lem}
	
	\begin{proof}
		Consider $\varepsilon > 0$. We have 
		\begin{align*}
			B'_n \cap \{\RT_{B_n} \geq \varepsilon/\mu(B'_n)\} & = \bigcup_{k\geq \varepsilon/\mu(B'_n)} \left(A \cap \{\RT_A >k\} \cap T^{-k}B_n \right) \\
			& \subset A \cap \{\RT_A \geq \varepsilon/\mu(B'_n)\}.
		\end{align*}
		Hence, 
		\begin{align*}
			\mu_{B'_n}(\RT_{B_n} \geq \varepsilon/\mu(B'_n)) \leq C \left(\frac{\varepsilon}{\mu(B'_n)}\right)^{-1/\beta} \cdot \frac{1}{\mu(B'_n)} \leq C\varepsilon^{-1/\beta} \mu(B'_n)^{1/\beta-1} \xrightarrow[n\to+\infty]{} 0,
		\end{align*}
		using the fact that $\beta < 1$ for $T \in \mathfrak{F}$.
	\end{proof}
	
	\noindent Let $\zeta \neq 0$. If $\zeta \in I_{\pm} \backslash\{0\}$, we will use the induction on $A_{\mp}$. In the following we might drop the indices once we choose $\zeta$.
	
	\begin{lem} \label{condition_2_thm}
		Let $\zeta \in \mathring{I_{\pm}}$. Then, we have $B_n \cap \{\RT_{B_n} \leq \RT_A\} = \emptyset$. If $\zeta = \pm 1$, we have $Q(B_n) \cap \{\RT_{B_n} \leq \RT_A\} = \emptyset$. 
	\end{lem}

	\noindent Let $\zeta \neq 0$.  If $\zeta \in I_{\pm}$, we set $A = \Delta_{0}^{\mp}$. We will characterize the shadow set $B'_n$ of $B_n$ in $A$.\\
	\noindent Let
	\[
	\Pre_k = \{\alpha \in A\; |\; T^k(\alpha) = \zeta \;\text{and}\; T^j(\alpha) \notin A\} \; \text{and set} \;  \Pre = \bigcup_{k\geq 1}\Pre_k.
	\]

	\noindent We will separate the proof in different parts. The easiest case is when $\zeta \notin \{0,-1,1\}$ for some $s \in \mathbb{N}$. Then, we will consider the special case $0$ and the fixed points $-1$ and $+1$. 
	
	\subsection{The case $\zeta \notin \{-1,0,1\}$.} In this section, we will assume that $\zeta \notin \{0,-1,1\}$ and prove the assertions in Theorem \ref{Dichotomy_double_intermittent} for such $\zeta$.
	
	\begin{lem} \label{carac_shadow_doubly}
		For every $k\geq 1$, we have $\Pre_k = \{\alpha_k\}$, $(\alpha_k)_{k\geq 1}$ is monotone and $\alpha_k \xrightarrow[n \to +\infty]{} 0$. Furthermore, for $n$ large enough, there exists a family of intervals $(B_n(\alpha_k))_{k\geq 1}$ included in $A$ such that
		\[B'_n = \bigcup_{k \geq 1} B'_n(\alpha_k).\]
		This is a disjoint union and $T^k(B'_n(\alpha_k)) = B_n$ is one to one and onto.
	\end{lem}	
	
	\begin{proof}

		In order to simplify the notation, we will consider $\zeta \in I_{-}$ and thus $A = \Delta_{0}^+$. The other case is identical and follows by symmetry.
		
		Consider $n$ large enough such that $B_n \subset \mathring{I_{-}}$. Recall that $B'_n = \bigcup_{k\geq 0} A \cap \{\RT_A > k\} \cap T^{-k}B_n$. For every $k \in \mathbb{N}$, we have $A \cap \{\RT_A > k\} \cap T^{-k}B_n = B'_n(\alpha_k)$, which we will ceck by induction. Indeed, for $k = 1$, $T^{-1}B_n = T_{-}^{-1}B_n \sqcup T_{+}^{-1}B_n$ and $T_{+}^{-1}B_n \subset A$ and $T_{+}^{-1}B_n$ is an interval containing $\alpha_k \in \Pre_1$. 
		
		By induction, we have $\{\RT_A > k\} \cap T^{k}B_n = T_{-}^{-k}B_n \sqcup T_{+}^{-1}(T_{-}^{-k+1}B_n)$ and of course $T_{+}^{-1}(T_{-}^{-k+1}B_n)\subset A$. The case $k = 1$ comes from the definition of $T$. Now, we have 
		\begin{align*}
			\{\RT_A > k+1\} \cap T^{-(k+1)}B_n &= T^{-1}\left(\{\RT_A > k\}\cap T^{-k}B_n \cap A^c\right) = T^{-1}\left(T_{-}^{-k}B_n\right) \\ &= T_{-}^{-(k+1)}B_n \sqcup T_{+}^{-1}(T_{-}^{-k}B_n).
		\end{align*}
		Now, we just have $T_{+}^{-1}(T_{-}^{-k+1}B_n) =: B'_n(\alpha_k)$ for $\alpha_k = T_{+}^{-1}(T_{-}^{-k+1}\zeta)$ and with $T^k (B'_n(\alpha_k)) = B_n$, bijectively. Furthermore, $T_{-}^{-k}\zeta \xrightarrow[k\to + \infty]{} -1$ meaning $\alpha_k \xrightarrow[k\to + \infty]{} 0$. We also have that since the sets $A \cap \{\RT_A > k\} \cap T^{-k}B_n$ are disjoint two by two for $n$ large enough, then so are the $B'_n(\alpha_k)$.
	\end{proof}
	
Since $B'_n$ consists of a countable union of intervals, we cannot apply directly decay of correlations to obtain the validity of the conditions and we have to truncate, as in \cite{AFFR17}. For that purpose, for every sequence $N(n)$ going to $+\infty$, we define $\tilde{B}'_n := \bigcup_{1\leq k \leq N(n)} B'_n(\alpha_k)$ the approximation of $B'_n$ with only $N(n)$ intervals. The following lemma tells us that this approximation is good enough to apply, without needing any further assumption on $N(n)$. Furthermore, for an interval $B'_n(\alpha_k)$, we will write $B'_n(\alpha_k)^-$ (resp. $B'_n(\alpha_k)^+$) for the left part (resp. the right part) of $B'_n(\alpha_k)$ stopping at $\alpha_k$ (resp. starting at $\alpha_k$) and procede likewise for $B_n$ around $\zeta$.
	
	\begin{lem} \label{truncation_approximation}
		For every sequence $(N(n))_{n\in \mathbb{N}}$, we have 
		\begin{equation*}
			\lim_{n\to +\infty} \frac{\mu\left(B'_n \backslash \tilde{B}'_n\right)}{\mu(B'_n)} = 0.
		\end{equation*}
	\end{lem}
	
	\begin{proof}
		Here, again, we will consider $\zeta \in I_{-}$ and use the induction on $A = \Delta_{0}^+$. The proof for $\zeta \in I_{+}$ is identical: one only need the switch the $+$ and $-$ signs. We split the proof in three different cases. 
		
		Assume first $T^{s}\zeta \in \mathring{A}$ (take $s$ minimal). Consider $n$ sufficiently large so that $T^s(B_n) \subset A$. Then, $B'_n \subset T_A^{-1}(T^s(B_n))$. Since $(A, T_A, \mu_A)$ is Rychlik and Markov, it has bounded distortion on each $(I_j)_{j \in \mathbb{N}}$. Furthermore, by construction, for every $k \geq 1$, we have that $B'_n(\alpha_k) \subset I_{j_k}$ for some $I_{j_k}$ such that $k(j_k) = q+k$. Thus, we have
		\begin{align*}
			\mu(B'_n(\alpha_k)) &\leq K\mu(I_{j_k})\mu(T_A(B'_n(\alpha_k))) \\
			& \leq K\mu(I_{j_k})\mu(T^s(B_n)) \leq C\mu(I_{j_k})\mu(B_n),
		\end{align*}
		Using the fact that $\mu \asymp \leb$ away from $-1$ and $1$ and $(T_{|B_n}^s)'$ is bounded away from $+\infty$ for $n$ large enough since $\zeta$ is not a preimage of $0$, we get
		\begin{align*}
			\mu\left(B'_n \backslash \tilde{B}'_n\right) & = \mu\left( \bigcup_{k > N(n)} B'_n(\alpha_k)\right) \leq C\sum_{k : k(j_k) > N(n) + s} \mu(I_{j_k}) \mu(B_n) \\
			& \leq C\mu(B_n)\mu\left(A \cap \{\RT_A > N(n)+s\}\right) = o(\mu(B_n)), 
		\end{align*}
		as soon as we take $N(n) \to +\infty$. Since we have $\mu(B'_n) = \mu(B_n)$, the result follows.
		
		Now, if $T^s \zeta = T^{-1}_{+}(0) \in \partial A$, we have $T^s(B^-_n) \subset A$ and thus, using the same method, we get 
		\begin{align*}
			\mu(B'_n(\alpha_k)^-) \leq C\mu(I_{j_k})\mu(B^-_n).
		\end{align*}
		Moreover, we also have $T^{s+1}(B^+_n) \subset A$ and the same argument holds, giving
		\begin{align*}
			\mu(B'_n(\alpha_k)^+) \leq C\mu(I_{j_k})\mu(B^+_n).
		\end{align*}
		Combining the two and summing, we obtain
		\begin{align*}
			\mu\left(B'_n \backslash \tilde{B}'_n\right) = o(\mu(B_n)), 
		\end{align*}
		as we take $N(n) \to +\infty$. \\
		\noindent The final case is when $T^s \zeta = 0$ with $T^{s-1}\zeta = T_{-}^{-1}(0)$. In particular, it means that $T^i\zeta \in I_{-}\backslash \Delta_{0}^-$ for every $i \in \{0,\dots,s-1\}$. Here again, we have $T^s B_n^+ \subset A$ leading to 
		\begin{align*}
			\mu(B'_n(\alpha_k)^+) \leq C\mu(I_{j_k})\mu(B^+_n).
		\end{align*}
		But now, $B_n^-$ will be sent close to $1$ and thus, when it comes back close in $A$, it will cover it preventing from obtaining the product by $B_n$. To avoid this issue, we will use the fact that $(\Delta_{0}^-, T_{\Delta_{0}^{-}}, \mu_{\Delta_{0}^-})$ is also Markov and Rychlik. We have $T^sB_n^- \subset \Delta_{0}^{-}$. For every $B_n(\alpha_k)$, consider $T_{-}^{-1}B_n(\alpha_k)^- \subset  \Delta_{0}^{-}$. We have $T_{-}^{-1}B_n(\alpha_k)^- \subset I'_{j_k}$ with $k(j_k) = k + s + 1$ (where $(I'_j)_j$ are the domain of the Rychlik map induced on $\Delta_{0}^-$). So, by bounded distortion again
		\begin{align*}
			\mu(T_{-}^{-1}B_n(\alpha_k)^-) &\leq K\mu(I'_{j_k})\mu(T_A(B'_n(\alpha_k))) \\
			& \leq K\mu(I_{j_k})\mu(T^s(B_n)) \leq C\mu(I'_{j_k})\mu(B_n),
		\end{align*}
		Now, using this time that $T'$ is bounded away from $0$ on $T_{-}^{-1}B_n(\alpha_k)^-$ for $n$ large enough (since the set is bounded away from $0$), we get $\mu(T_{-}^{-1}B_n(\alpha_k)^-) \geq c\mu(B_n(\alpha_k)^-)$. Combining again the two, we get the estimate
		\begin{align*}
			& \mu\left(B'_n \backslash \tilde{B}'_n\right) \\
			& \leq C\mu(B_n)\left(\mu\left(A \cap \{\RT_A > N(n)+s\}\right) + \mu\left(\Delta_{0}^- \cap \{\RT_{\Delta_{0}^-} > N(n)+s\}\right) \right) \\
			&= o(\mu(B_n)),
		\end{align*}
		as we take $N(n) \to +\infty$.
	\end{proof}

	\noindent Now we are able to finish the proof of Theorem \ref{Dichotomy_double_intermittent} for every $\zeta \notin \{-1,0, 1\}$.

	If $\zeta$ is not periodic, let $q = 0$. If $\zeta$ is periodic of period $p$, we consider $q = \left|\mathcal{O}(\zeta) \cap A\right|$ (note that the special cases of Lemma \ref{truncation_approximation} are not periodic points and we will have $q = 0$ in this case).
	By Lemma \ref{carac_shadow_doubly}, we have seen that $B'_n$ consists of a countable number of intervals. Using the arguments applied in  \cite[Theorems 4.3 and 4.4]{AFFR17}, we need to find a sequence $N(n) \in \mathbb{N}$ with $\lim_{n\to +\infty} N(n) = +\infty$ and $N(n) = o(n)$ such that 
	\begin{equation*}
		\lim_{n\to +\infty} \frac{\mu\left(B'_n \backslash \tilde{B}'_n\right)}{\mu(B'_n)} = 0.
	\end{equation*}
	and 
	\begin{enumerate}
		\item $\lim_{n\to + \infty} \|\mathds{1}_{Q(\tilde{B}'_n)}\|_{\mathcal{C}_1} n\rho_{t_n} =0$ for some sequence $(t_n)_{n\in \mathbb{N}}$ such that $t_n = o(n)$.
		\item $\lim_{n\to+\infty} \|\mathds{1}_{Q(\tilde{B}'_n)}\|_{\mathcal{C}_1} \sum_{j =R_n}^{+\infty} \rho_j = 0$.
	\end{enumerate}
	Here $Q(\tilde{B}'_n) = \tilde{B}'_n \cap  \bigcap_{i = 1}^q T_A^{-i} (\tilde{B}'_n)^c$. 	
	We have $\|Q(\tilde{B}'_n)\|_{BV} \leq 4N(n)+1$ (for every $k \geq 1$, $Q(\tilde{B}'_n) \cap B'_n(\alpha_k)$ consist of two disjoint intervals).
	
	 Since we do not have any constraint on $N(n)$, the three conditions are satisfied for a good sequence $N(n)$, thus giving the result and the convergence for $B'_n$. Corollary \ref{when_hits_the_reference_set} gives the result for $B_n$ since its conditions are satisfied by Lemma \ref{condition_1_thm}.
	
	\subsection{The case $\zeta = 0$.}
	\noindent For $\zeta = 0$, we will induce on $A := \Delta_0^+$. The proof would be identical if we had chosen $\Delta_0^-$. Again, we write $B_n^{\pm} = B_n \cap I_{\pm}$. The main issue here is that $\mu(B_n \cap A) = \mu(B_n^+) > 0$ and $\mu(B_n \cap A^c) = \mu(B_n^+) > 0$ for every $n \geq 0$. 
	
	\begin{lem} \label{condition_2_zeta=0}
		We have $\mu_{B_n}(\RT_{B_n} < \RT_A) \xrightarrow[n\to +\infty]{} 0$.
	\end{lem}
	
	\begin{proof}
		By the choice of $A$, we have $\mu_{B_n}(\RT_{B_n} < \RT_A) = \mu_{B_n}(B_n^+ \cap \{\RT_{B_n} < \RT_A\}) = \mu_{B_n}(B_n^+ \cap \{\RT_{B_n^-} < \RT_A\}) = \mu_{B_n}((B_n^-)' \cap B_n^+)$. \\
		$(B_n^-)' = \bigcup_{k\geq 1} (B_n^-)'(\alpha_k)$. Here again, we will use the same trick as in the proof of Lemma \ref{truncation_approximation} by considering the induced system $(\Delta_{0}^-, T_{\Delta_{0}^-}, \mu_{\Delta_{0}^-})$ and using that for every $k \geq 1$, $\mu(T_{-}^{-1}(B'_n)^-(\alpha_k)) \geq c\mu((B'_n)^-(\alpha_k))$ since the preimage is bounded away from $0$. Moreover, $\mu(T_{-}^{-1}(B'_n)^-(\alpha_k)) \subset I'_{j_k}$ for some $j_k$ with $k(j_k) = k + 1$. Hence, by bounded distortion
		\begin{align*}
			\mu((B_n^-)'(\alpha_k)) \leq C\mu(T_{-}^{-1}(B_n^-)'(\alpha_k)) \leq C\mu(I'_{j_k})\mu(B_n).
		\end{align*}
		Moreover, $T_{-}^{-1} B_n^+ \subset \bigcup_{j:k(j) > N(n)} I'_j$ for some $N(n) \xrightarrow[n\to + \infty]{} +\infty$ because $B_n^+$ is shrinking to $0+$. Hence,
		\begin{align*}
			\mu((B_n^-)' \cap B_n^+) \leq C\mu_{\Delta_0^-}(\RT_{\Delta_0^-} > N(n))\mu(B_n) = o(\mu(B_n)).
		\end{align*}
	\end{proof}
	
	\noindent Now, Lemma \ref{condition_2_zeta=0} allows us to use $B'_n$ without making further adjustments as in Corollary~\ref{when_does_not_hit}. The proof for $\zeta = 0$ is now similar to the case $\zeta \notin \{-1,0,1\}$ applying the same truncations. Hence the 
	REPP associated to $B_n$ converges to a standard homogeneous Poisson process.
	
	\subsection{The case $\zeta \in \{-1,1\}$.} 
	
	\noindent For $\zeta = -1$, we will consider the induction on $\Delta_{0}^+$ and for $\zeta = 1$, the induction on $\Delta_{0}^-$. The two cases are again symmetrical so we only consider the case $\zeta = -1$ and $A = \Delta_{0}^+$. Since $\zeta \in \partial I$, $B_n$ is only defined on one side of $\zeta$.
	
	\begin{lem} \label{carac_shadow_neutral}
		We have $\Pre_1 = \{0+\}$ and $\Pre_k = \emptyset$ for $k \geq 2$. We have 
		\begin{equation*}
			B'_n = T_{+}^{-1}B_n.
		\end{equation*}	
	\end{lem}
	
	\begin{proof}
		This comes from the fact that $T_{-}^{-1}B_n \subset B_n$ thus $B'_n = A\cap \{\RT_A > 1\} \cap T^{-1}B_n = T_{+}^{-1}B_n$.
	\end{proof}
	
	\noindent This case is easier since $B'_n$ is also an interval so there is no need to make approximation nor truncation. We directly get the convergence of the REPP associated to $B'_n$ and thus to the one associated to $Q(B_n) = B_n \cap T^{-1}
	B_n = B_n \backslash T_{-}^{-1}B_n$ to a homogeneous Poisson process of intensity one by Corollary~\ref{when_does_not_hit}, the conditions again coming from Lemma \ref{condition_1_thm} and \ref{condition_2_thm}. However, we cannot use the 
	reconstruction as in the Misiurewicz case because $\theta = \lim_{n\to + \infty} \mu(Q(B_n))/\mu(B_n) = 0$. However, one can still get the convergence of the HTS with the normalisation $\mu(Q(B_n))$ instead of $\mu(B_n)$ since the first hitting time is not affected by the infinite cluster. In this case, this is the same argument as 
	\cite[Theorem 5.1]{Zwe18} for Manneville-Pommeau maps.

	\bibliographystyle{amsalpha}
	
	\bibliography{bibliografia-v2.bib}

\providecommand{\bysame}{\leavevmode\hbox to3em{\hrulefill}\thinspace}
\providecommand{\MR}{\relax\ifhmode\unskip\space\fi MR }
\providecommand{\MRhref}[2]{%
  \href{http://www.ams.org/mathscinet-getitem?mr=#1}{#2}
}
\providecommand{\href}[2]{#2}
\begin{thebibliography}{AFFR17}

\bibitem[AAG21]{AAG21}
Miguel Abadi, Vitor Amorim, and Sandro Gallo, \emph{Potential well in
  {P}oincar\'{e} recurrence}, Entropy \textbf{23} (2021), no.~3, Paper No. 379,
  26. \MR{4234422}

\bibitem[Aba01]{Aba01}
Miguel Abadi, \emph{Exponential approximation for hitting times in mixing
  processes}, Math. Phys. Electron. J. \textbf{7} (2001), Paper 2, 19.
  \MR{1871384}

\bibitem[AFFR17]{AFFR17}
Davide Azevedo, Ana Cristina~Moreira Freitas, Jorge~Milhazes Freitas, and
  Fagner~B. Rodrigues, \emph{Extreme value laws for dynamical systems with
  countable extremal sets}, J. Stat. Phys. \textbf{167} (2017), no.~5,
  1244--1261. \MR{3647060}

\bibitem[AFV15]{AFV15}
Hale Ayta\c{c}, Jorge~Milhazes Freitas, and Sandro Vaienti, \emph{Laws of rare
  events for deterministic and random dynamical systems}, Trans. Amer. Math.
  Soc. \textbf{367} (2015), no.~11, 8229--8278. \MR{3391915}

\bibitem[AN05]{AN05}
Jon Aaronson and Hitoshi Nakada, \emph{On the mixing coefficients of piecewise
  monotonic maps}, Israel J. Math. \textbf{148} (2005), 1--10, Probability in
  mathematics. \MR{2191221}

\bibitem[BS21]{BS21}
Viviane Baladi and Daniel Smania, \emph{Fractional susceptibility functions for
  the quadratic family: {M}isiurewicz-{T}hurston parameters}, Comm. Math. Phys.
  \textbf{385} (2021), no.~3, 1957--2007. \MR{4284005}

\bibitem[BSTV03]{BSTV03}
H.~Bruin, B.~Saussol, S.~Troubetzkoy, and S.~Vaienti, \emph{Return time
  statistics via inducing}, Ergodic Theory Dynam. Systems \textbf{23} (2003),
  no.~4, 991--1013. \MR{MR1997964 (2005a:37004)}

\bibitem[CLM22]{CLM22}
Douglas Coates, Stefano Luzzatto, and Muhammad Mubarak, \emph{Doubly
  intermittent full branch maps with critical points and singularities}, 2022,
  preprint.

\bibitem[DT21]{DT21}
Mark~F. Demers and Mike Todd, \emph{Asymptotic escape rates and limiting
  distributions for multimodal maps}, Ergodic Theory Dynam. Systems \textbf{41}
  (2021), no.~6, 1656--1705. \MR{4252206}

\bibitem[FFMa18]{FFM18}
Ana Cristina~Moreira Freitas, Jorge~Milhazes Freitas, and M\'{a}rio
  Magalh\~{a}es, \emph{Convergence of marked point processes of excesses for
  dynamical systems}, J. Eur. Math. Soc. (JEMS) \textbf{20} (2018), no.~9,
  2131--2179. \MR{3836843}

\bibitem[FFT10]{FFT10}
Ana Cristina~Moreira Freitas, Jorge~Milhazes Freitas, and Mike Todd,
  \emph{Hitting time statistics and extreme value theory}, Probab. Theory
  Related Fields \textbf{147} (2010), no.~3-4, 675--710. \MR{2639719
  (2011g:37015)}

\bibitem[FFT12]{FFT12}
\bysame, \emph{The extremal index, hitting time statistics and periodicity},
  Adv. Math. \textbf{231} (2012), no.~5, 2626--2665. \MR{2970462}

\bibitem[FFT13]{FFT13}
\bysame, \emph{The compound {P}oisson limit ruling periodic extreme behaviour
  of non-uniformly hyperbolic dynamics}, Comm. Math. Phys. \textbf{321} (2013),
  no.~2, 483--527. \MR{3063917}

\bibitem[FFTV16]{FFTV16}
Ana Cristina~Moreira Freitas, Jorge~Milhazes Freitas, Mike Todd, and Sandro
  Vaienti, \emph{Rare events for the {M}anneville-{P}omeau map}, Stochastic
  Process. Appl. \textbf{126} (2016), no.~11, 3463--3479. \MR{3549714}

\bibitem[FP12]{FP12}
Andrew Ferguson and Mark Pollicott, \emph{Escape rates for gibbs measures},
  Ergod. Theory Dynam. Systems \textbf{32} (2012), 961--988.

\bibitem[HLV05]{HLV05}
N.~Haydn, Y.~Lacroix, and S.~Vaienti, \emph{Hitting and return times in ergodic
  dynamical systems}, Ann. Probab. \textbf{33} (2005), no.~5, 2043--2050.
  \MR{2165587}

\bibitem[HLV07]{HLV07}
N.~Haydn, E.~Lunedei, and S.~Vaienti, \emph{Averaged number of visits}, Chaos
  \textbf{17} (2007), no.~3, 033119, 13. \MR{2356973}

\bibitem[HP14]{HP14}
N.~T.~A. Haydn and Y.~Psiloyenis, \emph{Return times distribution for {M}arkov
  towers with decay of correlations}, Nonlinearity \textbf{27} (2014), no.~6,
  1323--1349. \MR{3215837}

\bibitem[HWZ14]{HWZ14}
Nicolai T.~A. Haydn, Nicole Winterberg, and Roland Zweim\"{u}ller,
  \emph{Return-time statistics, hitting-time statistics and inducing}, Ergodic
  theory, open dynamics, and coherent structures, Springer Proc. Math. Stat.,
  vol.~70, Springer, New York, 2014, pp.~217--227. \MR{3213501}

\bibitem[Kal21]{Kal21}
Olav Kallenberg, \emph{Foundations of modern probability}, Probability Theory
  and Stochastic Modelling, vol.~99, Springer, Cham, [2021] \copyright 2021,
  Third edition [of 1464694]. \MR{4226142}

\bibitem[Kel12]{K12}
Gerhard Keller, \emph{Rare events, exponential hitting times and extremal
  indices via spectral perturbation}, Dyn. Syst. \textbf{27} (2012), no.~1,
  11--27. \MR{2903242}

\bibitem[KY21]{KY21}
Yuri Kifer and Fan Yang, \emph{Geometric law for numbers of returns until a
  hazard under {$\varphi$}-mixing}, Israel J. Math. \textbf{244} (2021), no.~1,
  319--357. \MR{4344031}

\bibitem[LSV99]{LSV99}
Carlangelo Liverani, Beno{\^{\i}}t Saussol, and Sandro Vaienti, \emph{A
  probabilistic approach to intermittency}, Ergodic Theory Dynam. Systems
  \textbf{19} (1999), no.~3, 671--685. \MR{MR1695915 (2000d:37029)}

\bibitem[Mis81]{M81}
Micha{\l} Misiurewicz, \emph{Absolutely continuous measures for certain maps of
  an interval}, Inst. Hautes \'Etudes Sci. Publ. Math. \textbf{53} (1981),
  17--51. \MR{623533 (83j:58072)}

\bibitem[MS93]{MS93}
Welington~de Melo and Sebastian Strien, \emph{One-dimensional dynamics},
  Ergebnisse der Mathematik und ihrer Grenzgebiete. 3. Folge / A Series of
  Modern Surveys in Mathematics, Springer Berlin, Heidelberg, 1993.

\bibitem[Res08]{R08}
Sidney~I. Resnick, \emph{Extreme values, regular variation and point
  processes}, Springer Series in Operations Research and Financial Engineering,
  Springer, New York, 2008, Reprint of the 1987 original. \MR{2364939
  (2008h:60002)}

\bibitem[Rue09]{Rue09}
David Ruelle, \emph{Structure and {$f$}-dependence of the {A}.{C}.{I}.{M}. for
  a unimodal map {$f$} is {M}isiurewicz type}, Comm. Math. Phys. \textbf{287}
  (2009), no.~3, 1039--1070. \MR{2486672}

\bibitem[Ryc83]{R83}
Marek Rychlik, \emph{Bounded variation and invariant measures}, Studia Math.
  \textbf{76} (1983), no.~1, 69--80. \MR{MR728198 (85h:28019)}

\bibitem[Zwe16]{Zwe16}
Roland Zweim\"{u}ller, \emph{The general asymptotic return-time process},
  Israel J. Math. \textbf{212} (2016), no.~1, 1--36. \MR{3504316}

\bibitem[Zwe19]{Zwe18}
\bysame, \emph{Hitting-time limits for some exceptional rare events of ergodic
  maps}, Stochastic Process. Appl. \textbf{129} (2019), no.~5, 1556--1567.
  \MR{3944776}

\bibitem[Zwe22]{Z22}
\bysame, \emph{Hitting times and positions in rare events}, Ann. H. Lebesgue
  \textbf{5} (2022), 1361--1415. \MR{4526257}

\end{thebibliography}

\end{document}